\documentclass{amsart}
\usepackage{amsthm, bbold, amssymb, enumerate}
\DeclareMathOperator\Gal{Gal}
\DeclareMathOperator\Lie{Lie}

\DeclareMathOperator{\ab}{ab}
\usepackage[usenames, dvipsnames]{color}
\definecolor{mygray}{gray}{0.6}

\newcommand{\isom}{\cong}

\def\1{\mathbb{1}}

\def\cC{\mathcal{C}}

\def\calH{\mathcal{H}}

\def\cO{\mathcal{O}}

\def\cZ{\mathcal{Z}}

\def\comod{\backslash}

\def\A{\mathbb{A}}
\def\F{\mathbb{F}}
\def\Z{\mathbb{Z}}

\def\Q{\mathbb{Q}}

\def\C{\mathbb{C}}

\def\T{\mathbf{T}}

\def\U{\mathbf{U}}
\def\G{\mathbf{G}}
\def\H{\mathbf{H}}

\def\Res{\text{Res}}

\def\comod{\backslash}

\newtheorem{lemma}{Lemma}

\newtheorem{theorem}{Theorem}
\newtheorem{corollary}{Corollary}

\newtheorem{remark}{Remark}

\usepackage[centertags]{amsmath}
\usepackage{amsfonts}
\usepackage{amssymb}
\usepackage{amsthm}
\usepackage{amscd}
\usepackage{hyperref}
\usepackage{xypic}

\DeclareFontFamily{OT1}{rsfs}{}
\DeclareFontShape{OT1}{rsfs}{n}{it}{<-> rsfs10}{}
\DeclareMathAlphabet{\mathscr}{OT1}{rsfs}{n}{it}

\newcommand{\cA}{\mathcal{A}}
\newcommand{\cF}{\mathcal{F}}

\newcommand{\ds}{\displaystyle}

\newcommand{\cB}{\mathcal{B}}

\newcommand{\ra}{\rightarrow}
\newcommand{\xra}[1]{\xrightarrow{#1}}
\newcommand{\hra}{\hookrightarrow}

\newcommand{\mthree}[9]{\left [
        \begin{matrix}#1&#2&#3\\#4&#5&#6\\#7&#8&#9
        \end{matrix}\right ]}
\newcommand{\mtwo}[4]{\left [
        \begin{matrix}#1&#2\\#3&#4
        \end{matrix}\right ]}

\newcommand{\comment}[1]{}

\newcommand{\cS}{\mathcal{S}}

\newcommand{\bc}{\mathbf{c}}

\newcommand{\Qbar}{\overline{\Q}}

\newcommand{\Af}{\mathbf{A}_f}

\DeclareFontEncoding{OT2}{}{}

\renewcommand{\H}{\mathbf{H}}

\newcommand{\wbar}{\overline{w}}

\DeclareMathOperator{\Ad}{Ad}
\DeclareMathOperator{\Supp}{Supp}

\DeclareMathOperator{\dist}{dist}

\DeclareMathOperator{\Tr}{Tr}

\DeclareMathOperator{\Frob}{Frob}
\DeclareMathOperator{\Fr}{Fr}
\DeclareMathOperator{\Ver}{Ver}

\DeclareMathOperator{\ord}{ord}
\DeclareMathOperator{\GL}{\mathbf{GL}}

\DeclareMathOperator{\PGL}{\mathbf{PGL}}

\DeclareMathOperator{\Hbf}{\mathbf{H}}

\DeclareMathOperator{\Tbf}{\mathbf{T}}

\DeclareMathOperator{\Nbf}{\mathbf{N}}

\DeclareMathOperator{\diag}{diag}
\DeclareMathOperator{\Hom}{Hom}

\DeclareMathOperator{\Sh}{Sh}

\DeclareMathOperator{\inv}{\mathbf{inv}}

\DeclareMathOperator{\ur}{ur}

\DeclareMathOperator{\Stab}{Stab}

\DeclareMathOperator{\id}{id}

\DeclareMathOperator{\ad}{ad}
\DeclareMathOperator{\Art}{Art}

\DeclareMathOperator{\Hyp}{\mathbf{Hyp}}
\DeclareMathOperator{\Inv}{\mathbf{Inv}}

\theoremstyle{plain} 
\newtheorem{thm}{Theorem}[section] 
\newtheorem{thmx}{Theorem}

\newtheorem{prop}[thm]{Proposition}

\newtheorem{lem}[thm]{Lemma}

\theoremstyle{definition} 

\theoremstyle{remark}

\newcommand{\leftexp}[2]{{\vphantom{#2}}^{#1}{#2}}

\newcounter{tasknumber}
\newcommand{\task}[2][]{%
  \addtocounter{tasknumber}{1}%
  \begin{center}%
  \framebox[1.1\width]{\begin{minipage}{0.9\textwidth}%
  \textbf{Task \arabic{tasknumber}} \textit{\if!#1(unassigned)!\else (#1)\fi}: {#2}%
  \end{minipage}}%
  \end{center}%
}

\newcounter{assumptionnumber}
\newcommand{\assumption}[2][]{%
  \addtocounter{assumptionnumber}{1}%
  \begin{center}%
  \framebox[1.1\width]{\begin{minipage}{0.9\textwidth}%
  \textbf{Assumption \arabic{assumptionnumber}} \textit{\if!#1!\else (#1)\fi}: {#2}%
  \end{minipage}}%
  \end{center}%
}

\usepackage{xypic}
\usepackage{pictexwd,dcpic}
\usepackage{algorithm}
\usepackage{algorithmic}
\usepackage[dvipsnames]{xcolor}
\usepackage{mathrsfs} 

\title[]{Horizontal Distribution Relations For Special Cycles on Unitary Shimura Varieties: Split Case}
 
\author{Reda M. Boumasmoud}
\email{reda.boumasmoud@epfl.ch}
\address{Ecole Polytechnique F\'ed\'erale de Lausanne, Switzerland} 
 
\author[E. Hunter Brooks]{Ernest Hunter Brooks}
\email{ernest.brooks@epfl.ch}
\address{Ecole Polytechnique F\'ed\'erale de Lausanne, Switzerland} 
 
\author{Dimitar P. Jetchev}
\email{dimitar.jetchev@epfl.ch}
\address{Ecole Polytechnique F\'ed\'erale de Lausanne, Switzerland}

\begin{document}

\begin{abstract} 
We study the local behavior of special cycles on Shimura varieties for $\U(2, 1) \times \U(1, 1)$ in the setting of the Gan-Gross-Prasad conjectures at primes $\tau$ of the totally real field of definition of the unitary spaces which are \emph{split} in the corresponding totally imaginary quadratic extension. We establish a local formula for their fields of definition, and prove a distribution relation between the Galois and Hecke actions on them. This complements work of \cite{jetchev:unitary} at inert primes, where the combinatorics of the formulas are reduced to calculations on the Bruhat--Tits trees, which in the split case must be replaced with higher-dimensional buildings. \end{abstract}

\maketitle

\section{Introduction}

\subsection{Special cycles on Shimura varieties}

The conjectures of Gan, Gross and Prasad \cite{gan-gross-prasad} involve cycles on Shimura varieties constructed from embeddings of reductive groups. This process may be thought of as a generalization of the construction of Heegner points (which arise from the embedding of a non-split rank one torus over $\Q$ into $\GL_2$). In the particular case of embeddings of unitary groups arising from embeddings of Hermitian spaces, these cycles have been extensively studied. For example, Howard has studied their intersection theory in \cite{howard:kr}, and first steps toward a Gross--Zagier formula for them appear in Zhang \cite{zhang:afl} and Rapoport, Terstiege and Zhang \cite{rapoport-terstiege-zhang}. The versions of these cycles defined in \cite{jetchev:unitary}, whose specific construction will be recalled shortly, is designed for constructing Euler systems. The latter control the behavior of $L$-functions for automorphic forms on unitary groups at the central point and allow to prove non-trivial results towards the Birch and Swinnerton-Dyer conjecture and its generalizations, the Bloch--Kato--Beilinson conjectures.

In \cite{jetchev:unitary}, an extensive local theory of these cycles is developed at a place $\tau$ of the totally real field $F$ of definition of the unitary groups, under the assumption that $\tau$ is inert in the totally imaginary quadratic extension $E$ of $F$ that splits the groups. This has many applications to Euler systems already; however, one would like to extend these results to the setting where $\tau$ is split, for two reasons. First, there are many interesting questions about the global arithmetic of these cycles that cannot be deduced from the calculations at inert primes alone; for example, one would like a ``Heegner hypothesis'' describing when one can find a cycle defined over the Hilbert class field of $E$. Second, even if one is only interested in local methods, preliminary studies of $p$-adic $L$-functions in the unitary setting indicate that both the inert and split cases will be necessary (see e.g. \cite{eischen-harris-li-skinner} or the work in progress by Skinner and Prasanna).

This paper develops the arithmetic theory for the special cycles at split places, where the groups identify with general linear groups. One new technical input is a generalization of the methods of \cite{jetchev:unitary}, which work only for 1-dimensional trees, to higher-dimensional Bruhat--Tits buildings. In addition, the global arithmetic of the cycles becomes substantially more complicated due to a non-trivial action of the Frobenius elements at primes above $\tau$ in $\Gal(E[1]/E)$ on the cycles, where $E[1]$ denotes the Hilbert class field of $E$. Our main results are Theorems A, B, and C below. 

Before stating these results, we fix our notation. As already described, let $F$ be a totally real field and $E/F$ be a totally imaginary quadratic extension, together with a fixed embedding into $\mathbb{C}$. Let $\tau$ be a finite place of $F$, \emph{split} in $E$, and fix an embedding $\iota_\tau \colon \overline{F} \hra \overline{F}_\tau$. We will write $w$ for the place of $E$ determined by this choice, and $\overline{w}$ for its conjugate. Abusing notation, we will also write $w$ for the prime above $w$ determined by this choice in any field extension of $E$ contained in the fixed algebraic closure. Pick a uniformizer $\varpi$ for $F_\tau = E_w$, write $q$ for the cardinality of the residue field, and write $p$ for the rational prime below $\tau$.  

Let $W \subset V$ be an embedding of $E$-hermitian spaces with signatures $(1,1)$ (resp. $(2,1)$) at the distinguished archimedean place and $(2, 0)$ (resp. $(3,0)$) at the other archimedean places. Write $D$ for the orthogonal complement of $W$ in $V$. We may assume without loss of generality (see \cite[\textsection 1.2.1]{jetchev:unitary}) that $D$ is anisotropic and contains a vector $e_D \in D$ for which $\langle e_D, e_D\rangle = 1$.

There are algebraic groups $\G = \Res_{F/\Q}(\U(V) \times \U(W))$ and $\H = \Res_{F/\Q} \U(W)$, and an embedding $\H \hookrightarrow \G$, described in Section \ref{sec:prelim}. In this same section, a particular compact subgroup $K \subset \G(\A_f)$ and Hermitian symmetric domain $X$ are chosen, which give rise to a 
3-dimensional Shimura variety $\Sh_K(\G, X)$ and a family $\cZ_K(\G, \H)$ of special $1$-cycles on this threefold. The cycles in $\cZ_K(\G, \H)$ are defined over abelian extensions of $E$. There is a surjective map $\cZ_K \colon \G(\A_f) \to \cZ_K(\G, \H)$ inducing a bijection 
$$
\Nbf_{\G}(\H)(\Q) \backslash \G(\Af) / K \simeq \cZ_K(\G, \H), 
$$
where $\Nbf_\G(\H)$ denotes the normalizer of $\H$ in $\G$ (see \cite[Lem.2.3]{jetchev:unitary}). 
\subsection{Local conductor formula}
For $g \in \G(\A_f)$, one would like to compute the field of definition of a special cycle $\mathcal{Z}_K(g)$. This problem may be broken into a global and a local component. Globally, one wishes to understand the question of the field of definition of a particular cycle (e.g., $\mathcal{Z}_K(1)$), and locally, one seeks a formula computing the field of definition of $\mathcal{Z}_K(gg_\tau)$ in terms of the field of definition of the cycle $\mathcal{Z}_K(g)$, where $g_\tau \in \U(V)(F_\tau) \times \U(W)(F_\tau)$. 

Set $G_\tau = \U(V)(F_\tau) \times \U(W)(F_\tau)$ and $H_\tau = \U(W)(F_\tau)$; we identify $G_\tau$ and $H_\tau$ with general linear groups as normalized in Section~\ref{normalizationgln}. The particular choice of $K$ determines a compact open subgroup $K_\tau \subset G_\tau$, and the action of the decomposition group at $\tau$ on $\mathcal{Z}_K(\G, \H)$ may be described in terms of the action of $H_\tau$ on $G_\tau /K_\tau$ (see Section~\ref{subsec:shimrec}). 

To state the conductor formula precisely, one thus needs to understand the orbits of $H_\tau$ on $G_\tau /K_\tau$. In Theorem \ref{thm:orbits}A, we show that the orbits are in bijection with a set of $5$-tuples of integers $(s, r, d, m, n)$, modulo a certain equivalence relation. The combinatorics of these invariants is somewhat involved, but they have a natural interpretation in terms of the Bruhat--Tits building for $\GL_3(F_\tau) \times \GL_2(F_\tau)$ (see Section \ref{algorithm} for an algorithm to compute the invariants).

To determine the completions of the field of definition of the cycle $\mathcal{Z}_K(g)$ at $\tau$, one must work at the places $w$ and $\overline{w}$ simultaneously. Write $E_\tau = E \otimes_F F_\tau = E_w \times E_{\wbar}$. If $L$ is the field of definition $\mathcal{Z}_K(g)$, then the \'etale algebra $L \otimes F_\tau$ is determined by its corresponding norm subgroup in $E_\tau^\times = E_w^\times \times E_{\wbar}^\times$. It can be shown (Lemma~\ref{lem:norm}) that this subgroup is always of the form $F_\tau^\times \cdot (\mathcal{O}_{F_\tau} + \varpi^{\bc_\tau(g)} \mathcal{O}_{E_\tau})^\times$ for a unique non-negative integer $\bc_\tau(g)$ called the \emph{local conductor} at $\tau$.

\begin{thmx}\label{theoremA}
Let $g = (g_V, g_W) \in G_\tau$ have invariants $(s, r, d, m, n)$. The local conductor at $\tau$ of $\mathcal{Z}_K(g)$ is then given by
$$
\bc_\tau(g) = \max\{\min\{ m-n, d-m+n \},0\}.
$$
\end{thmx}

The reader may observe that the invariants labeled $s$ and $r$ do not affect the field of definition of a cycle. These invariants arise from the centers of $H_\tau$ and $G_\tau$, suggesting that the arithmetic of the cycles is ultimately governed by the adjoint forms of $\G$ and $\H$. However, the embedding $W \hookrightarrow V$ of Hermitian spaces does not induce a map of algebraic groups modulo the center, so one is forced to keep track of the $r$ and $s$ invariants when an embedding of groups is necessary (as will occur in Section \ref{sec:split-distrel-inv}).

\subsubsection{Applications to global cycles}
Theorem \ref{theoremA} allows one to state the field of definition of global cycles in some special cases (the analogous problem for torus embeddings in $\GL_2$ is the field of definition of Heegner points on the modular curve of level $1$).  For an ideal $\mathfrak c \subset \cO_F$, consider the following abelian extensions:

\begin{enumerate}

\item $E[\mathfrak c]$ - the ring class field of conductor $\mathfrak c$ determined (via global class field theory) by the norm subgroup $E^\times \cdot \widehat{\cO_{\mathfrak c}}^\times \subset \widehat{E}^\times$, 

\item $E(\mathfrak c)$ - the subfield of $E[\mathfrak c]$ whose norm subgroup is 
$$
E^\times \cdot \widehat{F}^\times \cdot \widehat{\cO_{\mathfrak c}}^\times \subset \widehat{E}^\times
$$
\end{enumerate}

\noindent Here, $\cO_{\mathfrak c}$ stands for the order $\cO_F + \mathfrak c \cO_E \subset \cO_E$ of conductor $\mathfrak c$. Note that $E[\mathfrak c]$ and $E(\mathfrak c)$ are in general not the same. The explicit form of the local conductor formula shows that the global cycles are defined over completions of $E(\mathfrak c)$.

\begin{corollary}\label{cor:globalcycle} Suppose that $E/F$ is unramified outside infinity. Suppose further that the compact $K = K_V \times K_W$, where $K_V$ and $K_W$ are stabilizers of global self-dual lattices in $V$, resp. $W$, and that all primes above $2$ in $F$ are split in $E$. Then the cycle $\mathcal{Z}_K(1)$ is defined over $E(1)$. In this setting, for any ideal $\mathfrak{c}$ of $F$, there exists a cycle whose field of definition is $E(\mathfrak{c})$.
\end{corollary}
\begin{proof}
The choice of compact is such that Theorem \ref{theoremA} applies to every place of $F$ that is split in $E$, and the local conductor formula of \cite{jetchev:unitary} applies to every place $\tau$ that is inert (the restriction at $2$ comes from the assumption $p\neq 2$ in \emph{loc. cit.}, which is not needed in this paper). It follows from the definition of invariants that the local invariants at split places are $(0, 0, 0, 0, 0)$, and the local invariants at inert places (in the sense of \cite[Thm.1.2.i]{jetchev:unitary}) are $(0, 0)$, so that all local conductors are $0$ and thus $\mathcal{Z}_K(1)$ is defined over $E(1)$.

For $\tau$ an arbitrary place of $F$, and $c$ any non-negative integer, there are explicit elements $g_c \in G_\tau$ such that the local conductor at $v$ satisfies
$$
\bc_\tau(g_c) = c.
$$
At inert places, one takes $g_c = (g_V, g_W)$ with conductor $c$ (see \cite[Thm 1.2.i]{jetchev:unitary}) and at split places, one may take $g_c$ to be any canonical matrix as in Theorem \ref{thm:orbits}B whose local conductor is $c$. The result on arbitrary conductors follows by taking products. 
\end{proof}

As an additional corollary, in the setting of Corollary \ref{cor:globalcycle}, one attains the first non-conjectural example of a family of trace compatible cycles defined over the anticyclotomic extension of $E$ attached to an \emph{inert} place $w$.

\begin{corollary}\label{cor:vertical} Under the assumptions of Corollary \ref{cor:globalcycle}, for $w$ inert, for any choice of automorphic representation $\pi$ of $\G$, one has a family of cycles $y_n \in \Z[\cZ_K(\G, \H)]$, such that the field of definition of $y_n$ is $E(\tau^n)$, satisfying 
$$\Tr_{E(\tau^n)/E(\tau^{n-1})} y_n = y_{n-1}.$$
\end{corollary}
Indeed, this is a restatement of the main result of \cite[\textsection 5]{brooks-boumasmoud-jetchev}, since the field called $L$ in \emph{loc. cit.} may be replaced with $E(1)$ due to Corollary \ref{cor:globalcycle}. It is evident that a significant generalization of these arithmetic results would arise from a weakening of the restrictions of Corollary \ref{cor:globalcycle}, which would in turn arise from a proper treatment of ramified primes and/or non-maximal compacts.

\subsection{Hecke polynomial}
Let $\calH = \calH(G_\tau, K_\tau)$ denote the local Hecke algebra, that is, the algebra of locally 
constant $\Z$-valued $K$-bi-invariant functions on $G_\tau$. 
A separate arithmetic question concerns the natural action of $\calH$ on $\Z[\mathcal{Z}_K(\G, \H)]$ 
(see Section \ref{sec:split-hecke} for the definition of the action). As originally observed by Langlands, the crucial properties of this action are reflected in a polynomial $H_w(x) \in \mathcal{H}[q^{\pm 1/2}][z]$, the so-called Hecke polynomial. 
This polynomial, for which we follow \cite{blasius-rogawski:zeta} with minor modifications as in \cite[\textsection 1.2.3, \textsection 4]{jetchev:unitary}), is defined precisely in Section~\ref{sec:split-hecke}. Heuristically, it is a family of $\overline{\Q}$-coefficient polynomials, indexed by the space of automorphic representations $\pi$ of $\G(\Af)$, whose specialization at a given $\pi$ on $\G$ gives the Euler factor at $w$ for an $L$-function attached to $\pi$. 

The following theorem explicitly computes the Hecke polynomial:
\begin{thmx}\label{theoremB}
The Hecke polynomial is given by
$$
H_w(z) = z^6 + c_5z^5 + c_4z^4 + c_3 z^3 + c_2 z^2 + c_1 z + c_0,
$$
where the coefficients $c_0, \dots, c_5 \in\mathcal{H}[q^{\pm 1/2}]$ are explicitly given in Proposition~\ref{thm:heckeformula}. In particular, the coefficients of the polynomial are in $\mathcal{H}$ and not just in $\mathcal{H}[q^{\pm 1/2}]$.
\end{thmx} 
Note that the latter consequence could be deduced directly from the normalization of the Satake transform in \cite[\textsection 8]{gross:satake} without explicitly computing the $c_i$ (see Section \ref{sec:split-hecke}). 

Although this is essentially a local formula, so that one is tempted to view it as a result about general linear groups, the definition of the Hecke polynomial depends on the choice of cocharacter defining the Shimura datum, which is global. In particular, the polynomial of Theorem \ref{theoremB} cannot be deduced from the polynomials for $\GL_3$ and $\GL_2$ computed in \cite{blasius-rogawski:zeta}, where the choice of cocharacter is different.

\subsection{Horizontal distribution relation}
The preceding two theorems give a description of the local Galois and Hecke actions at $\tau$ on the free abelian group $\Z[\cZ_K(\G, \H)]$ generated by the set of special cycles. One would like to give a relation between these two actions. To do so, for any $\xi \in \Z[\cZ_K(\G, \H)]$, write $E(\xi)$ for the compositum of the fields of definitions of the cycles in the support of $\xi$ (see \cite[\textsection 1.2.4]{jetchev:unitary}). 

The formulas for the classical distribution relation for $\GL_2$ (i.e. the relation between $T_p$ and the trace operator on Heegner points) suggests the following ``horizontal distribution relation,'' which we prove (see the exposition in \cite[\textsection 1.2.1]{jetchev:unitary} for a more detailed historical motivation):

\begin{thmx}\label{theoremC}
Suppose that $\xi_0 \in \cZ_K(\G, \H)$ has invariants $(0, 0, 0, 0, 0)$ at $\tau$. Then there is a cycle $\xi \in \Z[\cZ_K(\G, \H)]$ that satisfies $E(\xi)_w = E(\tau)_w$, such that one has
$$
H_w(\Fr_w) \xi_0 = \Tr_{E(\tau)_w/E(1)_w} \xi.
$$
\end{thmx}

Note that one could instead state the theorem globally, at the cost of introducing an auxiliary field $L$ of definition of $\xi_0$ which is an extension of $E$ in which $\tau$ splits completely. In this case, one replaces the trace over the local Galois group with the trace over the corresponding decomposition group (see e.g. the formulation of \cite[\textsection 1]{brooks-boumasmoud-jetchev}). 

Theorem \ref{theoremC}, in light of \cite[Thm.1.6]{jetchev:unitary} and the classical results for Heegner points, provides strong evidence that some general result should hold which expresses horizontal distribution relations for arbitrary embeddings of classical groups in terms of Hecke polynomials evaluated on Frobenius, but we refrain from formulating a conjecture at this time.

One also expects a vertical distribution relation to hold, which, among other things, would allow one to remove the ``inert'' hypothesis from Corollary \ref{cor:vertical}. Proving such a relation is part of the dissertation project (in progress) of the first author.

\subsubsection{Applications to Kolyvagin systems and Iwasawa theory}
Relations between the Hecke and Galois actions of the type proven in Theorem~\ref{theoremC}, \cite[Thm.1.6]{jetchev:unitary} and \cite[Thm.1.2]{brooks-boumasmoud-jetchev} are the main relations in the 
construction of a novel Euler system for $\U(1,1) \hra \U(2, 1) \times \U(1,1)$ analogous to the Euler system of Heegner points due to Kolyvagin \cite{kolyvagin:euler_systems} (see also \cite{gross:kolyvagin}) and its associated Kolyvagin system \cite{howard:heeg}. The latter have been used for two purposes: 1) proving the Birch and Swinnerton-Dyer conjecture when the basic Heegner point is non-torsion (and hence, the Kolyvagin system is non-trivial); 2) proving one divisibility in the anti-cyclotomic main conjecture of Iwasawa theory \cite{bertolini:compositio}, \cite{howard:heeg}. Although the auxiliary (Kolyvagin) primes have been chosen to be inert in $E$, the distribution relations and arithmetic applications in 2) have been using split primes. In addition, other $p$-adic applications (central value formulas for anticyclotomic $p$-adic $L$-functions) such as the recent formula of Bertolini--Darmon--Prasanna \cite{bertolini-darmon-prasanna} and its generalizations in \cite{brooks:shimura} and \cite{liu-zhang-zhang}, use split primes as well. Generalizing the latter to the unitary setting above is work in progress by Skinner and Prasanna. 

\subsection{Organization of the paper}
As already mentioned, a technical obstacle in the split case is the higher-dimensionality of the relevant buildings. This intervenes heavily in the proof of all three main theorems. In Section~\ref{sec:prelim} we recall basic facts about unitary groups, Shimura varieties, special cycles, and the Galois action via reciprocity laws on connected components as well as how the computation of the latter reduces to a question about local invariants of $H_\tau$-orbits on $G_\tau/K_\tau$ (which is the set of hyperspecial vertices on the product of the buildings for $\U(V)(F_\tau)$ and $\U(W)(F_\tau)$). In Section~\ref{sec:split-orbits}, we classify these $H_\tau$-orbits. We apply this classification to Galois orbits and computations of local conductors in Section~\ref{sec:split-galois}, where Theorem A is proven. In order to deal with the non-trivial Frobenius element, we also state a variant of the classification (Theorem \ref{thm:orbits}B), in which the $H_\tau$-orbits are further partitioned into Frobenius orbits. In Section~\ref{sec:split-hecke}, we prove Theorem B. To compute the Hecke polynomial, we first compute its coefficients under the Satake transform as elements of the Hecke algebra of the maximal torus, then compute the full polynomial by inverting the Satake transform (the latter reduces to a combinatorial problem on the Bruhat--Tits buildings, thanks to the theory of canonical retractions). The proof of Theorem C begins in Section~\ref{sec:split-distrel-inv}, where the relations are instead established on the set of invariants defined in Theorem~\ref{thm:distrel-inv}. This allows us to finish, in Section~\ref{sec:split-distrel-cycles}, by comparing the action of the Hecke operators on the corresponding buildings to the building-theoretic interpretation of the Galois action established in Section~\ref{sec:split-galois}.

\section{Preliminaries}\label{sec:prelim}

\subsection{Shimura varieties for unitary groups and special cycles}
Let $K$ be a compact open subgroup of $\G(\A_f)$ satisfying the conditions of \cite[\textsection 2.1.4] {jetchev:unitary} and such that $\tau$ is allowable for $(\G, \H, K)$ as in Definition 1.1 of \emph{loc. cit.} (the residue characteristic is assumed odd in that definition, but this assumption is not needed at split places). These assumptions imply that
\begin{itemize}
\item One has a factorization $K = K_\tau \cdot K^{\tau}$ with $K_\tau \subset (\U(V) \times \U(W))(F_\tau)$ and $K^{(\tau)} \subset (\U(V) \times \U(W))(\A_{F, f}^{(\tau)})$.
\item One has $K_\tau = K_{V, \tau} \times K_{W, \tau}$ where for $\star \in \{ V, W \}$, $K_{\star, \tau}$ denotes a hyperspecial maximal compact subgroup of $G_{\star, \tau}$. 
\end{itemize}

Let $X$ be the Hermitian symmetric domain for $\G$ defined in \cite[\textsection 2.2.7-\textsection2.2.7]{jetchev:unitary}; then $X = X_V \times X_W$ where $X_\star$ is a Hermitian symmetric domain for $\Res_{F/\Q} \U(\star)$. Let $Y \subset X$ denote the diagonal image of $X_W$ in $X$. 
The Shimura datum $(\G, X)$ then gives rise to a connected Shimura variety $\Sh_K(\G, X)$. This is a $3$-fold with a model over $E$ (the canonical model), whose complex points are given by
\begin{eqnarray*}
\Sh_K(\G, X)(\C) &=& \G(\Q) \comod \G(\A_f) \times X / K \\
&=& (\U(V) \times \U(W))(F) \comod \U(V) \times \U(W)(\mathbb{A}_{F, f}) \times X / K.
\end{eqnarray*}
The $\mathbb{C}$-points of the cycle $\cZ_K(g)$ of the introduction are then given by
$$
\cZ_K(g) := [gK \times Y] \subset \Sh_K(\G, X)(\C).
$$
\subsection{Shimura reciprocity laws and Galois action on special cycles}\label{subsec:shimrec}
Consider the torus 
$$
\Tbf = \Res_{E / \Q} \mathbb{G}_m = \Res_{F/\Q} \Res_{E/F} \mathbb{G}_m
$$
over $\Q$. Let $\Tbf^1$ be the subtorus $\Res_{F/\Q} \U(1)$, where $\U(1)$ denotes the unique (up to isomorphism) non-trivial $1$-dimensional unitary group over $F$ splitting over $E$, i.e. the norm one subtorus of $\Res_{E/F} \mathbb{G}_m$. 

There is a determinant map $\det \colon \H \ra \T^1$ and a map $r \colon \Tbf \ra \Tbf^1$ given on $R$-points, for $R$ a $\Q$-algebra, by 
$r(x) = \overline{x} / x$. Letting 
$$
\Ver_{E / F} \colon \Gal(F^{\ab} / F) \ra \Gal(E^{\ab} / E)
$$ 
be the transfer map, write $E[\infty]$ for the abelian extension of $E$ determined by the image of $\Ver_{E /F}$. The Artin map $\Art_E \colon \Tbf(\Af) \ra \Gal(E^{\ab} / E)$ induces an isomorphism 
$\Art^1_E \colon \Tbf^1(\Af) / \Tbf^1(\Q) \xra{\sim} \Gal(E[\infty] / E)$. (We normalize the reciprocity map so that uniformizers correspond to geometric Frobenii.)
Writing $\H^1$ for the kernel of the determinant map, the quotient of $\Nbf_{\G}(\Hbf)(\Q) \Hbf(\Af)$ by the normal subgroup $\Nbf_{\G}(\Hbf)(\Q)\Hbf^1(\Af)$ is isomorphic to 
$$
\frac{\Nbf_{\G}(\Hbf)(\Q) \Hbf(\Af)}{\Nbf_{\G}(\Hbf)(\Q) \Hbf^1(\Af)} \isom \frac{\Tbf^1(\Af)}{\Tbf^1(\Q)} \xra{\sim} \Gal(E[\infty] / E), 
$$
where the last map is $\ds \Art_E^1 \colon \frac{\Tbf^1(\Af)}{\Tbf^1(\Q)} \xra{\sim} \Gal(E[\infty] / E)$. 
There is thus a map
$$
{\det}^{*} \colon \Nbf_{\G}(\Hbf)(\Q) \Hbf(\Af) \twoheadrightarrow \Tbf^1(\Af) / \Tbf^1(\Q)
$$ 
induced by the determinant map $\det$. It is shown in \cite[Lem.2.5]{jetchev:unitary}, as a consequence of the Shimura reciprocity law and a strong approximation argument, that for any $\sigma \in \Gal(E^{\ab} / E)$ and any element $h_\sigma \in \Nbf_{\G}(\Hbf)(\Q) \Hbf(\Af)$ that satisfies   
$$
\Art^1_E({\det}^*(h_\sigma)) = \sigma |_{E[\infty]},
$$
one has
\begin{equation}\label{eq:gal}
\cZ_K(g)^{\sigma} = \cZ_K(h_\sigma g). 
\end{equation}

\subsection{Normalization isomorphism in the split case}\label{normalizationgln}
Recall that $\tau$ is split and that $w$ and $\wbar$ are the two places of $E$ above $\tau$. In this case, 
$\U(V)(F_\tau) \times \U(W)(F_\tau)$ can be identified with $\GL(V_w) \times \GL(W_w)$.  To fix such an identification, for an arbitrary Hermitian $E$-space $\mathscr V$, let 
$$
\U(\mathscr V)(F_\tau) = \{g \in \GL(\mathscr V)(E \otimes_F F_\tau) \colon \langle gv, gw \rangle_\tau = \langle v, w\rangle_{\tau}, \ \forall v, w \in \mathscr V \otimes_F F_\tau \}, 
$$
where $\langle \,, \rangle_{\tau} = \langle \,, \rangle \otimes_F F_\tau$. Since 
\begin{equation}\label{eq:isomE}
E_\tau := E \otimes_F F_\tau = E_w \oplus E_{\wbar} \isom F_\tau \oplus F_\tau,
\end{equation}
where the action of complex conjugation on the left-hand side corresponds to the involution $(s, t) \mapsto (t, s)$ on the right-hand side, one has  
\begin{equation}\label{eq:Vdecomp}
\mathscr V \otimes_F F_\tau = \mathscr V_w \oplus \mathscr V_{\wbar}, 
\end{equation}
and 
\begin{equation}
\GL(\mathscr V)(E \otimes_F F_\tau) = \GL(\mathscr V_w) \times \GL(\mathscr V_{\wbar}) \isom \GL(\mathscr V)(F_\tau) \times \GL(\mathscr V)(F_\tau).
\end{equation}
For any $v_{1}, v_2 \in V \otimes_F F_\tau$, write $v_1 = v_{1,w} + v_{1, \wbar}$ and $v_2 = v_{2, w} + v_{2, \wbar}$ according to \eqref{eq:Vdecomp}. Then 
$$
\langle (g_1, g_2) v_{1}, (g_1, g_2) v_2 \rangle_{\tau} = \left ( \langle g_2 v_{2, \wbar}, g_1 v_{1, w} \rangle, \langle g_1 v_{1, {\wbar}}, g_2 v_{2, w}\rangle \right ), 
$$
and 
$$
\langle v_{1}, v_2 \rangle_{\tau} = \left ( \langle v_{2, \wbar}, v_{1, w} \rangle, \langle v_{1, {\wbar}}, v_{2, w}\rangle \right ), 
$$
and hence, $\langle g_2 v_{2, \wbar}, g_1 v_{1, w} \rangle = \langle v_{2, \wbar}, v_{1, w} \rangle$ for all $v_{1, w} \in V_w$ and $v_{2, \wbar} \in V_{\wbar}$. It follows that the map $g = (g_1, g_2) \mapsto g_1$ defines an isomorphism 
$$
\U(\mathscr V)(F_\tau) \isom \GL(\mathscr V_w).
$$ 
Recalling the fixed embedding $\iota_\tau \colon \overline{F} \hra \overline{F}_\tau$, if $w$ is the place of $E$ corresponding to this embedding, then there is no ambiguity in writing $\mathscr V_w$ as $\mathscr V_\tau$ and viewing it as a vector space over $E_w = F_\tau$. We will retain this notation throughout.

\subsection{Action of Frobenius on unramified special cycles}\label{subsec:frob}

If $\cZ_K(g)$ is defined over a number field that is unramified at $\tau$, then to describe $\Fr_w \cZ_K(g)$, where $\Fr_w$ is the geometric Frobenius, it suffices to take a matrix $h_{\Fr} \in \GL(W_\tau) \isom \U(W)(F_\tau)$ that satisfies $v(\det(h_{\Fr})) = -1$ (e.g., $\diag(\varpi^{-1}, 1)$). The description \eqref{eq:gal} of the action of the Galois group on special cycles together with the discussion in Section~\ref{normalizationgln} imply that  
\begin{equation}\label{eq:Frob}
\Fr_w \cdot \cZ_K(g) = \cZ_K(h_{\Fr} g), 
\end{equation}
where 
$h_{\Fr}$ denotes the image in $\G(\Af)$ under the natural embeddings 
$$
H_\tau = \U(W)(F_\tau) \hra \U(W)(\mathbb{A}_{F}) \hra \U(V)(\mathbb{A}_{F}) \times \U(W)(\mathbb{A}_{F}) \isom \G(\Af).
$$ 

\subsection{Orders in $E_\tau$ and filtrations}
Let the place $\tau$ and $w$ be as above. 
Under the identification~\eqref{eq:isomE}, the maximal order $\cO_{E_\tau}$ of the \'etale algebra $E_\tau$ is 
$$
\cO_{E_\tau} \isom \cO_{E_w} \oplus \cO_{E_{\overline{w}}} \isom \cO_{F_{\tau}} \oplus \cO_{F_{\tau}}.
$$  
Define $\cO_0 = \cO_{E_\tau}$ and more generally, $\cO_c = \cO_{F_\tau} + \varpi^c \cO_{E_\tau}$ where 
$\cO_{F_\tau} \subset \cO_{E_\tau}$ via the diagonal embedding and $\varpi$ is a uniformizer of $F_\tau$. There is thus a filtration
\begin{equation}\label{eq:filt1}
\cO_{E_\tau} =: \cO_0 \supset \cO_1 \supset \cO_2 \supset \dots. 
\end{equation}
The image of $E_\tau^\times$ under the map $r \colon x \mapsto \overline{x} / x$ is the group 
$$
\cO_0^1 := \{(x^{-1}, x) \in E_\tau^\times \colon x \in F_\tau^\times \}.
$$
In other words, $\cO_0^1$ is the image of the map $u \colon F_\tau^\times \ra E_\tau^\times$, $u(x) := (x^{-1}, x)$. 
The filtration \eqref{eq:filt1} induces a filtration 
$$
\cO_0^1 \supset \cO^1_1 \supset \cO^1_2 \supset \dots,  
$$
where $\cO^1_c = r(\cO_c^\times)$. It is not hard to check that $\cO_c^1 = u(1 + \varpi^c \cO_{F_\tau})$. This gives a filtration on $H_\tau$, setting $H_c := \det^{-1}(\cO^1_c)$. Via the isomorphism $H_\tau \isom \GL_2(F_\tau)$, this filtration corresponds to the filtration on $\GL_2(F_\tau)$ that is the preimage (under the determinant map) of the filtration 
on $\cO_{F_\tau}^\times$:  
$$
\cO_{F_\tau}^\times \supset 1 + \varpi \cO_{F_\tau} \supset 1 + \varpi^2 \cO_{F_\tau} \supset \dots. 
$$
Abusing notation, denote these subgroups of $H_\tau \isom \GL_2(F_\tau)$ by 
$$
H_0 \supset H_1 \supset \dots
$$ 

\begin{lem}\label{lem:norm}
One has $r^{-1} (\cO_c^1) = \varpi^{\Z} \cO_c^\times$. 
\end{lem}

\begin{proof}
Consider the commutative diagram
$$
\xymatrix{
E_{\tau}^{\times}  \ar@{<-^{}}[r]^{r} \ar@{<-^{)}}[d]_{u}& E_{\tau}^{\times}\ar@{->>}[dl]^\varphi \ar@{->>}[d]^{} \\ 
F_{\tau}^{\times} \ar@{<<-}[r]  &  E_{\tau}^{\times}/\varpi^{\Z}
}
$$
where $\varphi(x, y) =  xy^{-1}$, and $u(x) = (x^{-1}, x)$. Then
$$r^{-1}(\cO_c^1)=\varphi^{{-1}}(u^{-1}(\cO_{c}^1))=\left\{(x,y)\in E_{\tau}^{\times}: xy^{-1}\in 1 + \varpi^c \cO_{F_\tau}\right\},$$
For such a pair $(x,y)$, one has $\ord(x)=\ord(y)$. Write $\ord(x, y)$ for this common valuation. There there exist integers $m, n$ such that
$$ \varpi^{{-ord(x,y)}}x \in 1 + \varpi^n \cO_{F_\tau}\text{ and } \varpi^{{ord(x,y)}}y^{-1} \in 1 + \varpi^m \cO_{F_\tau}.$$
The element $(x,y)$ is then in $r^{-1}(\cO_c^1)$ if and only if $\min{(m,n)}=c$, if and only if $\varpi^{{-ord(x,y)}}x,\varpi^{{-ord(x,y)}}y \in 1 + \varpi^c \cO_{F_\tau}$.
\end{proof}

\section{The action of $H_\tau$ on $G_\tau/K_\tau$}\label{sec:split-orbits}
Since this section is local at $\tau$, we switch to a notation schema that suppresses localizations at $\tau$: write  $k_0 = F_\tau$ with uniformizer $\varpi$ and ring of integers $\cO$. Write $V$ for an arbitrary $3$-dimensional vector space over $k_0$ and $W$ for a $2$-dimensional subspace; we write $D$ for a line in $V$ which is not contained in $W$. Write $G_W = \GL(W)(k_0)$, $G_V = \GL(V)(k_0)$ and $G = G_V \times G_W$, and let $H$ be the group $G_W$ viewed as a subgroup of $G$ via the corresponding diagonal embedding. Let $H_0 = \{h \in H \colon v (\det(h)) = 0\}$ where $v \colon k_0 \ra \Z$ denotes the valuation. There is an exact sequence 
$$
0 \ra H_0 \ra H \xra{v \circ \det} \Z \ra 0. 
$$
Choose lattices $L_D$ and $L_W$ in $D$ and $W$, respectively, and set $L_V = L_W \oplus L_D$.  Write $K_V$ (resp., $K_W$), for the stabilizer of the homothety class of the lattice $L_V$ (resp., of the lattice $L_W$), in $G_V$ (resp., $G_W$). Finally, choose an $\cO$-basis $\{e_1, e_2\}$ of $W$ and an $\cO$-basis $\{e_3\}$ for $D$. This choice identifies $G_V \isom \GL_3(k_0)$ and $G_W$ with the subgroup of block-diagonal matrices of the form 
$
\ds \left \{
\begin{pmatrix}
\star & \star & 0 \\ 
\star & \star & 0 \\ 
0 & 0 & 1 
\end{pmatrix} 
\subset \GL_3(k_0) \right \}. 
$

\setcounter{paragraph}{0}
\subsection{Classification result}\label{subsec:Horbits}
For $\star \in \{V, W\}$, let $\Hyp_{\star} = G_\star/K_\star$ be the set of lattices in $\star$. If $L_V'$ is a lattice in $V$, we will say that $L_V'$ is \emph{adapted to the decomposition} $V = W \oplus D$, or just \emph{adapted}, if $L_V' = (L_V' \cap W) \oplus (L_V' \cap D)$. 

The following theorem classifies $H$-orbits and $H_0$-orbits on $\Hyp_V \times \Hyp_W$.

\begin{theorem}\label{thm:orbits}
{\bf [A.]} ($H$-orbits): The set of $H$-orbits on elements of $\Hyp_V \times \Hyp_W$ is in bijection with the set of quintuples $(s, r, d, m, n) \in \Z^2 \times \mathbb{Z}_{\geq 0}^3/\sim$, where the equivalence relation $\sim$ is given by
\begin{itemize}
\item $(s, r, d, m, n) \sim (s, r, d, 0, n) \text{ for } d \geq 0, m \leq n$.
\item $(s, r, d, m, n) \sim (s, r, d, m, 0) \text{ for } d \geq 0, m \geq n + d$.
\end{itemize}
A representative with invariants $(r, s, d, m, n)$ is the lattice determined by the column vectors of the matrices
$$
\varpi^s \begin{pmatrix}1 & 0 & \varpi^{-m} \\ 0 & 1 & \varpi^{-n} \\ 0 & 0 & 1 \end{pmatrix} \qquad \text{ and } \qquad \varpi^r \begin{pmatrix} \varpi^{d} & 0 \\ 0 & 1  \end{pmatrix}.
$$

\noindent {\bf [B.]} ($H_0$-orbits): The $H_0$-orbits contained in a given $H$-orbit corresponding to a quintuple $(s, r, d, m, n)$ are indexed by a parameter $k \in \Z$. A representative of invariant $(k, s, r, d, m, n)$ is explicitly given by the column vectors of the matrices
$$
\varpi^s \begin{pmatrix}\varpi^k & 0 & \varpi^{k-m} \\ 0 & 1 & \varpi^{-n} \\ 0 & 0 & 1 \end{pmatrix} \qquad \text{ and } \qquad \varpi^r \begin{pmatrix} \varpi^{k+d+r} & 0 \\ 0 & \varpi^r  \end{pmatrix}.
$$
\end{theorem}

The proof, presented in the remainder of this section, proceeds as follows: we will show that there is an $H$-equivariant map $\rho \colon \Hyp_V \to \Hyp_W$. Then, using the map $(\rho, \id): \Hyp_V \times \Hyp_W \to \Hyp_W \times \Hyp_W$, it will suffice to classify $H$ and $H_0$-orbits on the target and on the fiber over a convenient point of the target. An algorithm to compute the invariants given a pair of lattices is given at the end of this section. Let $P \subset G_V$ be the subgroup that stabilizes the subspace $W \subset V$. The choice of the basis $\{e_1, e_2, e_3\}$ identifies  $P$ with the following subgroup of $\GL_3(k_0)$:  
$$
\left\{ 
\begin{pmatrix} 
\star & \star & \star \\ 
\star & \star & \star \\ 
0& 0& \star
\end{pmatrix} 
\right\} \subset \GL_3(k_0).
$$
Consider the Levi decomposition
$P = U \ltimes (ZH)$
where $U$ is the unipotent radical of $P$, which is the group of matrices of the form $\begin{pmatrix} 1& 0& \star \\ 0& 1& \star \\ 0 &0& 1 \end{pmatrix},$ $Z = Z(\GL_3(k_0))$, and $H = G_W$ is embedded in $G_V$ as the subgroup 
$
\begin{pmatrix} \star& \star& 0 \\ \star& \star& 0 \\ 0 &0& 1 \end{pmatrix}.
$ 

\begin{lemma}\label{lem:projection}
There is an $H$-equivariant map of simplicial complexes $\rho: \Hyp_V \to \Hyp_W$ such that the fiber over $x_0$ is $UZx_0$.
\end{lemma}
\begin{proof}
For $x \in \Hyp_V$, we claim that $UZx \cap \Hyp_W$ is a singleton. For the special case $x = x_0$, this follows from a quick calculation. For the more general (but still special) case where $x$ is in $\Hyp_W$, one picks $h$ moving $x_0$ to $x$ and uses the fact that $UZ$ normalizes $H$ to deduce that $UZx \cap \Hyp_W = hUZh^{-1}hx_0 \cap h\Hyp_W = x$. The claim now follows for all $x$ because of the transitivity of the action of $G = PK$ on $\Hyp_V$. Let $\rho$ be the map sending $x$ to the unique $y \in UZx \cap \Hyp_W$. Then $\rho$ is $H$-equivariant: if $y \in UZx \cap \Hyp_W$, then $hy \in UZhx \cap \cB(G_W)$, which again follows because $UZ$ normalizes $H$.
\end{proof}

We next describe the $H$-orbits on $\Hyp_W \times \Hyp_W$. If $L_W'$ and $L_W''$ be two lattices in $W$, define their \emph{relative position} $\mu(L_W', L_W'')$ as follows: let $r$ be the minimal integer such that $\varpi^r L_W' \subset L_W''$. Then the quotient $L_W''/L_W'$ is cyclic of order $q^d$ for some integer $d \geq 0$, and the relative position  $\mu(L_W', L_W'')$ of $L_W'$ and $L_W''$ is the pair $(d, r)$.  Alternatively, $r$ and $d$ are the unique integers such that $ d \geq 0$ and there is a basis $\{e_1, e_2\}$ for $L_W'$ with $\{\varpi^re_1, \varpi^{r+d}e_2\}$ a basis for $L_W''$.  If $\mu(L_W', L_W'')  = (d, r)$, then $\mu(L_W'', L_W')  = (d, -d-r)$. 

The action of $H$ on $\Hyp_W \times \Hyp_W$ clearly preserves the relative position; conversely, the elementary divisors theorem implies that there is an element of $H$ moving any $(y_0, y_1) \in \Hyp_W \times \Hyp_W$ with relative position $(d, r)$ to $(x_0, x_{d, r})$, where $x_0 = \langle e_1, e_2 \rangle$ and $x_{d, r} = \varpi^r e_1, \varpi^{d+r} e_2$; it follows that $d$ and $r$ uniquely classify the $H$-orbits.

The stabilizer in $H$ of the pair $(x_0, x_{d, r})$, written with respect to the basis $\{e_1, e_2\}$, is $R_d^\times$, where
$$
R_d := \begin{pmatrix} \star & \star \\ \varpi^d\star & \star  \end{pmatrix} \subset M_2(\cO).
$$ 

The $R_d^\times$-orbits on $UZ x_0$ are the $R_d^\times$-orbits on
\[
UZ / (UZ) \cap K_V = \left\{ \begin{pmatrix} 
1 & 0 & \star \\ 
0 & 1 & \star \\ 
0& 0& 1
\end{pmatrix}: \star \in \cO/k_0 \right\},
\]
which is identified as an $R_d^\times$-module with $(\cO/k_0)^2$ with the usual left-action of $R_d^\times$. For $m, n \geq 0$, write
$$
v_{m,n} = \begin{pmatrix} \varpi^{-m} \\ \varpi^{-n} \end{pmatrix} \in (\cO/k_0)^2.
$$

\noindent The following proposition then concludes the proof of Theorem \ref{thm:orbits}A:
\begin{prop}\label{prop:RdOrbits}
The set of $R_d^\times$-orbits on $(k_0/\cO)^2$ is in bijection with pairs $(m, n) \in \Z_{\geq 0}^2$ under the following equivalence relations:
\begin{itemize}
\item $(m, n) \sim (0, n)$ whenever $m \leq n$
\item $(m, n) \sim (0, n)$ whenever $m \geq n + d$.
\end{itemize}
The class of $(m,n)$ is represented by the vector $v_{m, n}$.
\end{prop}

\begin{proof}
Given $v = (v_1, v_2) \in (k_0/\cO)^2$, set $m(v) = \min\{0, -\text{val}(v_1)\}$ and $n(v) = \min\{ 0, -\text{val}(v_w)\}$.  Any $v$ can be taken to $v_{m(v), n(v)}$ by a diagonal matrix in $R_d^\times$.

If $m - n \leq 0$, the upper unipotent matrix
$\ds \begin{pmatrix} 1 & -\varpi^{n-m} \\ 0 & 1 \end{pmatrix}$
moves $v_{m,n}$ to $\begin{pmatrix} 0 \\ \varpi^{-n} \end{pmatrix}$, which is the same element of $(k_0/\cO)^2$ as $v_{0, n}$. If $m - n \geq d$, the lower unipotent matrix 
$\ds \begin{pmatrix} 1 & 0 \\ -\varpi^{m-n} & 1\end{pmatrix}$ 
moves $v_{m,n}$ to
$\ds \begin{pmatrix} \varpi^{-m} \\ 0 \end{pmatrix}$, 
which is the same element of $(k_0/\cO)^2$ as $v_{m, 0}$. 

Now, if $d = 0$, at least one of these conditions on $m-n$ must hold, so using the above matrices and the flip-flop matrix
$\ds w = \begin{pmatrix} 0 & 1 \\ 1 & 0 \end{pmatrix}$, 
we see that the only $R_0^\times$-invariant of a vector is its order in the group $(k_0/\cO)^2$, i.e. a vector of order $m$ is necessarily equivalent to $v_{m, 0}$.

If $d > 0$, the above observations still show that the list in the statement of the proposition is exhaustive; we must show any two members of it are $R_d^\times$-inequivalent. A quick calculation shows that any element $\gamma \in R_d^\times$ can be written as a product $\gamma = LUT$, where $L, U \in R_d^\times$ are respectively lower and upper triangular unipotent matrices and $T \in R_d^\times$ is diagonal (we remark that this is not true when $d = 0$, $\gamma = \begin{pmatrix} \varpi & 1 \\ \varpi - 1 & 1 \end{pmatrix} $being a counterexample). The matrix $T$ cannot change $m(v)$ or $n(v)$, and one easily calculates that triangular unipotent matrices can only change them within the putative equivalence classes (it suffices to calculate on $v_{m, n}$, since diagonal matrices normalize the group of upper or lower triangular unipotent matrices), proving the lemma.
\end{proof}

\noindent Theorem \ref{thm:orbits}B now follows from the fact that the exact sequence 
$$
0 \to H_0 \to H \to \Z \to 0
$$
is split by the map $\Z \to H$ given by $1 \mapsto \delta:= \diag(\varpi, 1)$. Writing $(x_{s, m,n}, x_{r, d})$ for the standard pair of Theorem \ref{thm:orbits}A, it follows that the $H$-orbit of $(x_{s, m, n}, x_{r, d})$ is partitioned into disjoint $H_0$-orbits $(\delta^k x_{s, m, n}, \delta^k x_{r, d})$ (two such pairs corresponding to different $k$ cannot be $H_0$-equivalent, because $H_0$ preserves the valuation of the determinant of any matrix whose column vectors span a given lattice).

\subsection{Local Galois orbits}\label{sec:split-galois}
\setcounter{paragraph}{0}

For $(x, y) \in \Hyp_V \times \Hyp_W$, write $\inv_\tau(x, y) \in \Z^3 \times \Z_{\geq 0}^3/\sim$ for the invariants of Theorem \ref{thm:orbits}B.

\begin{theorem}[Local Conductor Formula]\label{thm:localcond}
Let $(x, y) \in \Hyp_V \times \Hyp_W$ , and suppose $\inv_\tau(x, y) = (k, s, r, d, n, m)$. If $m \leq n$ or $m \geq n + d$, the local conductor of a pair is $0$. Otherwise,  
$$
\bc_\tau(x, y) = \min \{(m-n, d - m + n) \}. 
$$
\end{theorem}

\begin{remark}The fact that the local conductor does not depend on $k$ is obvious a priori, since the local conductor depends on the $H$-action only; the fact that it does not depend on $r$ and $s$ will be deduced in the proof below, and ultimately follows from the independence of the stabilizer $R_d^\times$ of $(x_0, x_{d, r})$ from these invariants.
\end{remark}

Using the explicit list of representatives in Theorem \ref{thm:orbits}A and the identifications in the discussion above Proposition \ref{prop:RdOrbits}, we see that it suffices to compute the determinant of the stabilizer of 
$\ds v_{m, n} = \begin{pmatrix} \varpi^{-m} \\ \varpi^{-n} \end{pmatrix} \in (k_0 / \cO)^2$ as a subgroup of $R_d^\times$.

In the ``uninteresting'' cases where $m \leq n$ or $n + d \leq m$, the determinant map identifies the stabilizer of $v_{m, n}$ with $\cO^\times$: indeed, in these cases, by Proposition \ref{prop:RdOrbits} we have equivalences with $v_{0, n}$ (resp. $v_{m, 0}$), and for any $u \in \cO^\times$, such a vector is stabilized by $\diag(u, 1)$ (resp. $\diag(1, u)$). The local conductor formula then reduces to the following:

\begin{lem}\label{lem:intcase}
If $n < m < n+d$ then one has
$$
\det(\Stab_{G_W}(v_{m, n})) = 1+\varpi^{\min (m-n, d-m+n ) }\cO.
$$
\end{lem}

\begin{proof}
Let $M \in \Stab_{R_d}(x)$ and write
$$
M = \begin{pmatrix} \alpha & \beta \\ \varpi^d \gamma & \delta \end{pmatrix}.
$$
The condition that $M$ stabilizes $x$ is equivalent (under our assumptions on $m, n$) to
$$
\alpha \equiv 1 + \beta \varpi^{m-n} \mod \varpi^m \cO 
\qquad 
\text{ and }
\qquad 
\delta \equiv 1 + \gamma \varpi^{d-m+n} \mod \varpi^n \cO.
$$
These two conditions imply $\alpha\delta - \beta\gamma \equiv 1 \mod \min(\varpi^{m-n}, \varpi^{d-m+n})$ as desired. 

Conversely, it suffices to show that there exists a matrix in the stabilizer of $v_{m, n}$ with determinant $1 \mod \varpi^A$ for each of $A = d-m+n$ and $A = m-n$. One finds that the following two matrices work:
$$
\begin{pmatrix} 1 & 0 \\ \varpi^d & 1-\varpi^{d-m+n} \end{pmatrix}, \ \begin{pmatrix} 1-\varpi^{m-n} & 1 \\ 0 & 1\end{pmatrix}.
$$
\end{proof}

\section{The Hecke Polynomial in the Split Case}\label{sec:split-hecke}
We retain the local notation from Section~\ref{sec:split-orbits}. 
Let $q$ be the size of the residue field of $k_0$. The basis $\{e_1, e_2, e_3\}$ yields maximal tori $T_V$ and $T_W$ for $G_V$ and $G_W$, respectively (the diagonal tori under the identifications $G_V \isom \GL_3(k_0)$ and $G_W \isom \GL_2(k_0)$), as well as the standard Borel subgroups $B_V \subset G_V$ and $B_W \subset G_W$ of upper-diagonal matrices. Let $B = B_V \times B_W$, $T = T_V \times T_W \subset G$ and $T_c = T \cap K$. Let $\Omega(T) := N_G(T) / T$ be the Weyl group.  
For any ring $R$ denote by $\calH_R(G,K)$ the Hecke algebra of the pair $(G,K)$ with $R$-coefficients, i.e. the ring of $R$-valued $K$-bi-invariant locally constant functions $G \to R$ with compact support (the addition structure is inherited from $R$, but the multiplication operation is convolution). Let $\calH = \calH_{\Z}(G, K)$ and let 
$\calH_R = \calH \otimes_{\Z} R$.

\subsection{Unramified representations}
For $\star \in \{V, W\}$, recall that an irreducible, admissible automorphic representation $\pi_{\star}$ of $G_{\star}$ is called unramified if $\pi_{\star}^{K_{\star}} \ne 0$. Let $\Pi_{\ur}(G_{\star})$ be the set of isomorphism classes of unramified representations. Similarly, we define $\Pi_{\ur}(T_{\star})$ as well as the set of isomorphism classes of unramified representations $\Pi_{\ur}(G)$ for the product group and the compact $K = K_V \times K_W$ (and $\Pi_{\ur}(T)$ as well). We will use the fact that $\Pi_{\ur}(G) \isom \Pi_{\ur}(T) / {\Omega(T)}$ and that unramified irreducible representations of $T$ correspond to characters $\xi \colon T \ra \C^\times$ containing $T_c = T \cap K$ in the kernel. 
Hence, $\Pi_{\ur}(T) = \Hom(T/T_c, \C^\times)$. Under the map
$$
T \ra \Hom(X^*(T), \Z), \qquad t \mapsto \{\alpha \ra \text{val}(\alpha(t))\}. 
$$
one then has an identification
\begin{equation}\label{eq:unrep}
\Pi_{\ur}(T) \isom \Hom(X_*(T), \C^\times) \isom \widehat{T}.
\end{equation}
Furthermore, there is an injective homomorphism of $\C$-algebras 
$$
\calH_{\C}(G, K) \hra \text{Hom}(\Pi_{\ur}(G), \C), \qquad \mathbf{1}_{KgK}  \mapsto \{\pi \mapsto \Tr(\mathbf{1}_{KgK}) | \pi^K\}. 
$$
Satake \cite{satake:spherical} computes the image of $\calH_{\C}(G, K)$ under this map. More 
precisely, the isomorphism $\Pi_{\ur}(G) \isom \Pi_{\ur}(T) / {\Omega(T)}$ yields an identification of $\Hom(\Pi_{\ur}(G), \C)$ with $\Hom(\Pi_{\ur}(T), \C)^{\Omega(T)})$. If one further uses the identification \eqref{eq:unrep}, one obtains an identification 
\begin{equation}\label{hecke:functions}
\calH_{\C}(G, K) \isom \C[\widehat{T}]^{\Omega(T)}.
\end{equation}
where $\C[\widehat{T}]$ denotes algebraic functions on the dual torus $\widehat{T} = \Hom(X_*(T), \C^\times)$ (note that the situation is simplified from that in \cite{satake:spherical} as $T$ is split). Similarly, there is an identification 
\begin{equation}\label{eq:hecketorus}
\calH_{\C}(T, T_c) \isom \C[\widehat{T}]. 
\end{equation}

\subsection{Satake isomorphism}
Using the identifications from the previous section, there is an isomorphism $\calH \otimes \Z[q^{\pm 1/2}] \isom \calH(T, T_c)^{\Omega(T)} \otimes \Z[q^{\pm 1/2}]$ (see~\cite{satake:spherical}). Consider the following commutative diagram: 
$$
\xymatrix{
\calH(G, K) \otimes \Z[q^{\pm 1/2}] \ar@{^{(}->}[r]^{\widehat{}} \ar@{->}[d]^{|_{B}}& \calH(T, T_c) \otimes \Z[q^{\pm 1/2}] \\ 
\calH(B, L) \otimes \Z[q^{\pm 1/2}] \ar@{->}[r]^{\mathcal S}  &  \calH(T, T_c) \otimes \Z[q^{\pm 1/2}]\ar@{->}[u]_{|\delta|^{1/2}}. 
}
$$
Here, $L = B \cap K$ and $|_{B} \colon \calH(G, K) \ra \calH(B, L)$ is the restriction of functions, 
the map $\cS \colon \calH(B, L) \ra \calH(T, T_c)$ is determined by the rule $\cS(\mathbf{1}_{L g L })  = [L \cap gL g^{-1}] \mathbf{1}_{g T_c}$ for $g \in T$, i.e., it is obtained by taking quotients by the unipotent radical. Moreover, $\delta \colon T \ra q^{\Z}$ is the character defined as follows: 
$$
\delta(t) = \left | \det (\Ad(t)|_{\Lie(U)} ) \right |, 
$$
where $U \subset B$ is the unipotent radical and $\Ad(t)$ denotes the automorphism of 
$\Lie(U)$ defined by the adjoint representation (see~\cite{gross:satake} and 
\cite[\textsection 1.2]{wedhorn:congruence}). Here, $|\cdot|$ is normalized so that $|\varpi| = q^{-1}$.
In fact, the above diagram gives an algebra homomorphism $\cS \circ |_{B} \colon \calH(G, K) \ra \calH(T, T_c)$ (the one inducing the usual Satake isomorphism), and a twisted version 
$$
|\delta|^{1/2} \circ \cS \circ |_{B} \colon \calH(G, K) \otimes \Z[q^{\pm 1/2}]\ra \calH(T, T_c) \otimes \Z[q^{\pm 1/2}]
$$ 
which we denote by $\widehat{\cdot}$.

Finally, in the case when $\delta^{1/2}$ takes values in the subgroup $q^{\Z}$, one has an isomorphism 
$$
\calH \otimes_{\Z} \Z[q^{\pm 1}] \isom \calH(T, T_c)^{\Omega(T)} \otimes_{\Z} \Z[q^{\pm 1}].
$$

\subsection{Computing the polynomial with coefficients in $\mathcal H(T, T_c)^{\Omega(T)}$}\label{subsec:heckepoltorus}

As in \cite[\textsection 2.2.7]{jetchev:unitary}, let $\mathcal C(\G_{\Qbar})$ be the set of conjugacy classes of cocharacters $\G_{\Qbar}$ and let $\mu \in \mathcal C(\G_{\Qbar})$ be the conjugacy class associated to the Shimura datum $(\G, X)$. We may assume that our representative $\mu$ of the conjugacy class $[\mu]$ takes values in the diagonal torus $T$; then $\mu$ corresponds to a character $\widehat{\mu} \colon \widehat{T} \ra \C^\times$. There is a unique character of $\widehat{T}$ in the $\Omega(\widehat{T})$-orbit of $\widehat{\mu}$ that is dominant with respect to the Borel pair $(\widehat{B}, \widehat{T})$, where $\widehat{B}$ denotes the product of the standard Borels of upper triangular matrices in $\widehat{G} = \GL_3(\C) \times \GL_2(\C)$. The complex representation $r \colon \widehat{G} \to\GL_{6}(\C)$ of $\widehat{G}$ whose highest weight is this dominant character of $\widehat{T}$ is given by
$$
g = (A_V,A_W)\mapsto \leftexp{t}{A_V^{-1}}\otimes\leftexp{t}{A_W^{-1}}, \qquad A_V \in \GL_3(\C),\ A_W \in \GL_{2}(\C).
$$ 
This representation extends uniquely to a representation of ${}^L G  = \widehat{G} \times \Gal(\overline{\Q}/\Q)$, also denoted $r$.

Following Blasius and Rogawski \cite[\textsection 6]{blasius-rogawski:zeta} as well as the remark in \cite[\textsection 1.2.3]{jetchev:unitary}, one may 
associate to the place $w$ of the reflex field $E$ of 
$(\G, X)$ the polynomial 
\begin{equation}\label{eq:maxtorus}
H_w(z)=\det\left(z-q^{\frac{\dim X}{2}} r(g \times \Phi_\tau) \right ) \in \C[\widehat{G}] [z],
\end{equation}
where the ring of coefficients is the ring of algebraic functions on $\widehat{G}$ and $\Phi_\tau \in W_\tau$ denotes the Frobenius in ${}^L G$ (which acts trivially in our case).

We may use the identifications of the preceding section to map $H_w(z)$ to polynomials in $\C[\widehat{T}][z]$ and $\calH_\C(T, T_c)[z]$, which, abusing notation, we will continue to denote by $H_w(z)$. This requires restricting
the representation $r \colon \widehat{G} \ra \GL_6(\C)$ to the dual torus $\widehat{T}$, evaluating the determinant in \eqref{eq:maxtorus}, writing the coefficients as algebraic functions on  $\widehat{T}$, and then using the identification \eqref{eq:hecketorus}. For the restriction step, letting $A_V=\diag{(x_1, x_2 ,x_3)} \in \widehat{T}_V$ and $A_W=\diag{(y_1, y_{2})} \in \widehat{T}_W$, the polynomial identifies with 
\begin{equation}\label{eq:heckepolmaxtorus}
H_w(z)=\prod_{i=1}^3\prod_{j=1}^{2}\left(z-q^{3/2} x_i^{-1}y_j^{-1}\right) \in \C[\widehat{T}][z]. 
\end{equation}
Here, $x_i^{-1} y_j^{-1}$ denotes the function $x_i^{-1} y_j^{-1} \colon \widehat{T} \ra \C$ sending  
$$
(\diag{(x_1, x_2, x_3)}, \diag{(y_1, y_{2})}) \mapsto x_i^{-1}y_j^{-1}. 
$$  
Under the identification \eqref{eq:hecketorus}, this function corresponds to the element $\mathbf{1}_{({g_i^{-1}},{h_j^{-1}})T_c} \in \mathcal H(T, T_c)$ where 
$$
g_i= \diag{(\underbrace{1, \dots 1}_{i-1}, \varpi, 1 \dots,1)}\text{ and } h_j=\diag (\underbrace{1 ,\dots 1}_{j-1}, \varpi, 1 ,\dots,1).
$$
Since the Weyl group $\Omega(\widehat{T})$ acts by permuting the $x_i$'s and $y_j$'s and 
since the right-hand side of \eqref{eq:heckepolmaxtorus} is symmetric on both the $x_i$'s and the $y_j$'s, it is plain that $H_w(z) \in \C[\widehat{T}]^{\Omega(T)}[z]$.  
Now, $\C[\widehat{T}]^{\Omega(T)} \isom \calH(T, T_c)^{\Omega(T)}$ via \eqref{eq:hecketorus}. Let 
$$
s_{0,1}=\mathbf{1}_{(1,h_1)T_c}+\mathbf{1}_{(1,h_2)T_c} \in \mathcal{H}(T,T_c)^{\Omega(T)}.
$$
and 
$$
s_{0,-1}=\mathbf{1}_{(1,h_1^{-1})T_c}+\mathbf{1}_{(1,h_2^{-1})T_c} \in \mathcal{H}(T,T_c)^{\Omega(T)}. 
$$
Let 
$$
s_{1,0}=\mathbf{1}_{(g_1,1) T_c}+\mathbf{1}_{(g_2,1) T_c}+\mathbf{1}_{(g_3,1) T_c} \in 
\mathcal{H}(T,T_c)^{\Omega(T)},
$$ 
and 
$$
s_{-1,0}=\mathbf{1}_{(g_{1}^{-1},1) T_c}+\mathbf{1}_{(g_2^{-1},1) T_c}+\mathbf{1}_{(g_3^{-1},1) T_c}  \in \mathcal{H}(T,T_c)^{\Omega(T)}.  
$$
In addition, let $u_V = \mathbf{1}_{(\diag(\varpi, \varpi, \varpi), 1)T_c}$ and let $u_W = \mathbf{1}_{(1, \diag(\varpi, \varpi))T_c}$. One calculates that
\begin{align*}
H_w(z)&= z^6 - \\
&-q^{3/2} s_{-1,0}s_{0, -1}z^5 + \\
&+ q^3(s_{1,0} s_{0, -1}^2 u_V^{-1} + s_{-1,0}^2u_W^{-1} - 2s_{1,0}u_V^{-1}u_W^{-1})z^4 - \\
&- q^{9/2} (s_{0, -1}^3 u_V^{-1} - 3 s_{0, -1} u_V^{-1} u_W^{-1} + s_{1,0} s_{-1,0} s_{0, -1} u_V^{-1} u_W^{-1})z^3 + \\
&+ q^6(s_{1,0}^2 u_V^{-2} u_W^{-2}  + s_{-1,0}s_{0, -1}^2 u_V^{-1}u_W^{-1} - 2s_{-1,0}u_V^{-1}u_W^{-2})z^2 - \\
&- q^{15/2} s_{1,0}s_{0, -1}u_V^{-2} u_W^{-2} z + \\
&+ q^9 u_V^{-2} u_W^{-3}, 
\end{align*}
when viewed as a polynomial in $\mathcal{H}_\C(T, T_c)^{\Omega(T)}[x]$, and moreover that this polynomial in fact has coefficients in the subalgebra $\mathcal{H}_{\Z[q^{\pm 1/2}]}(T,T_c)^{\Omega(T)}[z]$.
We wish to express the coefficients of this polynomial as elements of $\mathcal H(G,K)$, i.e. to invert the Satake isomorphism $\mathcal H(G, K) \xra{\widehat \cdot } \mathcal{H}(T,T_c)^{\Omega(T)}$; this final polynomial, which we will continue to denote $H_w(z)$, is the Hecke polynomial of the introduction. This may be done via a building-theoretic approach. To simplify the calculations, first observe that there is an isomorphism
\begin{equation}\label{eq:glpgl}
\mathcal H(G_V, K_V) \isom \mathcal H(G_V^{\ad}, K_V^{\ad})[u, u^{-1}], 
\end{equation}
where $u$ denotes a formal variable, that yields (using the previously chosen bases) an isomorphism 
$$
\mathcal H(\GL_3(k_0), \GL_3(\cO)) \isom \mathcal H(\PGL_3(k_0), \PGL_3(\cO))[u, u^{-1}].
$$

\subsection{The building for $\PGL_n(k_0)$}
The projective linear groups $\PGL_n(k_0)$ over $p$-adic fields have associated polysimplical complexes (Bruhat--Tits buildings) described in detail in \cite{bruhat-tits:3}. For the purposes of this paper, we only recall certain features of these buildings that will be used and leave the reader to consult \cite{bruhat-tits:3} for a complete treatment (see also \cite{goldman-iwahori} for an alternative interpretation in terms of $p$-adic norms).   

\subsubsection{Hyperspecial points}
Let $\mathscr V$ be a $k_0$-vector space of dimension $n$ and consider the Bruhat--Tits building 
$\cB(\PGL(\mathscr V))$ of the $p$-adic groups 
$\PGL(\mathscr V)$. 
The hyperspecial points 
$\Hyp_{\PGL(\mathscr V)} \subset \cB(\PGL(\mathscr V))$ are bijection with $k_0^\times$-homothety classes $[L]$ of $\cO$-lattices $L \subset \mathscr V$. 
 
\subsubsection{Apartments, facets and chambers} 
A framing of $\mathscr V$ is a set $\mathcal F = \{\ell_1, \dots, \ell_n\}$ of $k_0$-lines $\ell_i$ that span $\mathscr V$. Each framing determines a decomposition of $k_0$-vector spaces  
\begin{equation}\label{eq:decomp}
\mathscr V = \ell_1 \oplus \dots \oplus \ell_n. 
\end{equation}
We call a lattice $L \subset \mathscr V$ \emph{adapted} to the decomposition \eqref{eq:decomp} if 
$$
L = (\ell_1 \cap L) \oplus \dots \oplus (\ell_n \cap L).  
$$
Clearly, the property of an $\cO$-lattice of being adapted to a decomposition is invariant under $k_0$-homothety. 
Given a framing $\cF$, the set of homothety classes of $\cO$-lattices adapted to the  decomposition determined by that framing forms an apartment $\cA_{\PGL(\mathscr V)}(\cF)$ of $\cB(\PGL(\mathscr V))$. Conversely, each apartment of 
of $\cB(\PGL(\mathscr V))$ arises from some framing $\mathcal F$ of $\mathscr V$. 
Chambers ($n-1$-simplices) of $\cB(\PGL(\mathscr V))$ correspond to chains of full-rank $\cO$-lattices 
$$
L_0 \subsetneq L_1 \subsetneq \dots \subsetneq L_{n-1} \subsetneq \varpi^{-1} L_0
$$
where each successive quotient is 1-dimensional. 
More generally, for $r \leq n$, $r$-faces (or $r$-simplices) of $\cB(\PGL(\mathscr V))$ correspond to chains  
$$
L_0 \subsetneq L_1 \subsetneq L_2 \subsetneq \dots \subsetneq L_r \subsetneq \varpi^{-1} L_0. 
$$
We refer to $(n-2)$-faces as facets. Two chambers $\cC'$ and $\cC''$ are called adjacent if they share a common facet.  
 
\subsubsection{Adjacency relations and relative positions}
Given two full-rank $\cO$ lattices $L'$ and $L''$ of $\mathscr V$, the theory of elementary divisors yields a $k_0$-basis $\{e_1', \dots, e_n'\}$ of $\mathfrak V$ together with integers $a_1 \geq a_2 \geq \dots \geq a_n$ such that 
$$
L' = \cO e_1' \oplus \dots \oplus \cO e_n' \qquad \text{and} \qquad L'' = \varpi^{a_1} e_1' \oplus \dots \oplus \varpi^{a_n} e_n'. 
$$
Even if the basis $\{e_1', \dots, e_n'\}$ need not be unique, the (ordered) $n$-tuple $(a_1, \dots, a_n)$ is uniquely determined by the lattices $L'$ and $L''$ and we refer to it 
as the relative position of $L'$ and $L''$ and denote it by 
$$
\{L'' : L'\} = (a_1, \dots, a_n). 
$$
Similarly, if $[L']$ and $[L'']$ are two homothety classes of full-rank $\cO$-lattices, we can always find representatives $L'$ and $L''$ 
such that 
$$
\{L'' : L'\} = (a_1, \dots, a_{n-1}, 0), \qquad a_1 \geq a_2 \geq \dots \geq a_{n-1} \geq 0, 
$$
and the $(n-1)$-tuple $(a_1, \dots, a_{n-1})$ will be uniquely determined from $[L']$ and $[L'']$. We thus denote it by 
$$
\{[L''] : [L']\} = (a_1, \dots, a_{n-1}). 
$$ 

\subsubsection{Gallery distance on the buildings and on apartments}
The main notion of a distance function on a building that we will be using is the \emph{gallery distance}. Given any two chambers $\cC'$ and $\cC''$ of (of either one of the buildings), we define the gallery distance $\dist(\cC', \cC'')$ as the minimal non-negative integer $n$ for which there exists a 
gallery $\cC_0, \cC_1, \dots, \cC_n$ such that $\cC_0 = \cC'$ and $\cC_n = \cC''$ (recall that a gallery is a sequence of chambers so that for any $i = 0, 1, \dots, n-1$, $\cC_i$ and 
$\cC_{i+1}$ are adjacent chambers). More generally, if $x$ is any point on the corresponding building and $\cC'$ is any chamber then we define the distance $\dist(x, \cC')$ as the minimal $n$ for which there exists a gallery $\cC_0, \dots, \cC_n$ such that $x \in \cC_0$ and $\cC_n = \cC'$.  

\subsubsection{Canonical retractions}
Let $\mathcal A$ be an apartment of $\cB(\PGL(\mathscr V))$ and let $\mathcal C \subset \mathcal A$ be a chamber. One can show \cite[\textsection 4.2]{garrett:buildings}
that there exists a retraction map $\rho_{\cA, \cC} \colon \cB(\PGL(\mathscr V)) \ra \cA$ satisfying the following properties: 
\begin{enumerate}
\item For any chamber $\mathcal D$ of $\mathscr A$ and a facet $x \in \cC$, 
$$
\dist(x, \mathcal D) = \dist_{\cA}(x, \mathcal D), 
$$
where $\dist$ denotes the gallery distance function on $\cB(\PGL(\mathscr V))$ and $\dist_{\cA}$ denotes the gallery distance on the 
apartment $\cA$. 

\item When restricted to any other apartment $\cA'$ containing $\cC$, the map $\rho|_{\cA} \colon \cA' \ra \cA$ is the identity 
map on the intersection $\cA' \cap \cA$. 

\item The map $\rho_{\cA, \cC}$ is the unique map of simplicial complexes that fixes the chamber $\cC$ pointwise and such that 
for each hyperspecial point $x$ of $\cC$ and each chamber $\mathcal D$ of $\cB(\PGL(\mathscr V))$, 
$$
\dist(x, \mathcal D) = \dist(x, \rho_{\cA, \cC}(\mathcal D)).
$$
\end{enumerate}
 
\subsection{The sub-building $\cB(\GL(\mathscr W)) \subset \cB(\PGL(\mathscr V))$}
Let $\mathscr W \subset \mathscr V$ be a codimension 1 $k_0$-vector subspace and let $\mathscr D$ be a complement, i.e., a $k_0$-line such that $\mathscr V = \mathscr W \oplus \mathscr D$. Suppose that $L_{\mathscr D} \subset \mathscr D$ is a fixed $\cO$-lattice. Once these are fixed, there is a way of defining a sub-building $\cB(\GL(\mathscr W))$ of the building $\cB(\PGL(\mathscr V))$ as follows: given any $\cO$-lattice $L_{\mathscr W}$, we view that as a hyperspecial vertex in $\cB(\PGL(\mathscr V))$ by considering the homothety class $[L_{\mathscr W} \oplus L_{\mathscr D}]$. The subcomplex $\cB(\GL(\mathscr W))$ (having the same dimension as the building $\cB(\PGL(\mathscr V))$) then inherits all the metric properties from the building $\cB(\PGL(\mathscr V))$. In particular, the 
apartments of $\cB(\GL(\mathscr W))$ correspond to the framings for $\mathscr V$ containing $\mathscr D$ as one of the framing lines. Chambers for the building $\cB(\GL(\mathscr W))$ will then correspond to chambers of the building $\cB(\PGL(\mathscr V))$ whose vertices are hyperspecial points of $\cB(\GL(\mathscr W))$. Finally, canonical retraction maps for $\cB(\GL(\mathscr W))$ are naturally inherited from the retraction maps for $\cB(\PGL(\mathscr V))$.  
 
\subsection{Inversion of $\mathcal H_{\Z[q^{\pm 1/2}]}(G, K) \rightarrow \mathcal H_{\Z[q^{\pm 1/2}]}(T, T_c)^{\Omega(T)}$}\label{notation:inverse}

Let $x_0$ be the homothety class $[L_V]$ of the lattice $L_V = \langle e_1, e_2, e_3\rangle$ which is a hyperspecial point on the building of $G_V^{\ad} \isom \PGL_3(k_0)$. Let $y_0$ be the lattice $\langle e_1, e_2 \rangle$ and view it as a hyperspecial point on $G_W$. Let $K_V^{\ad} \times K_W \subset G_V^{\ad} \times G_W$ be the image of $K$ under the surjection $G \twoheadrightarrow G_V^{\ad} \times G_W$. In other words, $K_V^{\ad}$ is the stabilizer (in $G_V^{\ad}$) of $x_0$ and $K_W$ is the stabilizer of $y_0$ in $G_W$. 

Since $\mathcal H(G, K) \isom \mathcal H(G_V^{\ad} \times G_W, K_V^{\ad} \times K_W)[u, u^{-1}]$ and $u \mapsto \mathbf{1}_{(\diag(\varpi, \varpi, \varpi), 1)T_c} = u_V$ under the Satake isomorphism, it suffices to invert the Satake isomorphism on 
$\mathcal H_{\Z[q^{\pm 1/2}]}(G_V^{\ad}, K_V^{\ad})  \otimes \mathcal H_{\Z[q^{\pm 1/2}]}(G_W, K_W)$. 

Let $t_V=\text{diag}(\varpi^{a_1}, \varpi^{a_{2}}, 1)$ and $t_W = \text{diag}(\varpi^{b_1}, \varpi^{b_{2}})$ such that $a_1 \ge a_2 \geq 0$ and $b_1 \ge b_{2}$. Then

\begin{equation}\label{eq:hecke-adjV}
\mathbf{1}_{K_V^{\ad} t_V K_V^{\ad}} (x_0) =   \sum_{\substack{x \in \cB(G_V^{\ad})\\  \{x : x_{0}\}= \{q^{a_1}, q^{a_{2}}, 1\}}} (x). 
\end{equation}
and 
\begin{equation}\label{eq:hecke-adjW}
\mathbf{1}_{K_W t_W K_W} (y_0) = \sum_{\substack{y \in \cB(G_W)\\  \{y : y_0\}= \{q^{b_1}, q^{b_2}\}}} (y). 
\end{equation}

Let $\cA_V$ be the apartment of $\cB(G_V^{\ad})$ determined by the fixed basis $\left \{e_1, e_2, e_3\right \}$ and let $\cA_W$ be the apartment of $\cB(G_W)$ determined by $\left \{e_1, e_{2}\right \}$. Following the idea in \cite[Prop.4.2]{jetchev:unitary}, we fix chambers $\cC_V$ and $\cC_W$ of $\cA_V$ and $\cA_W$ containing $x_0$ and $y_0$, respectively and we first compute $\cS \circ |_{B_V^{\ad}} \mathbf{1}_{K_V^{\ad} t_V K_V^{\ad}}$ and $\cS \circ |_{B_W} \mathbf{1}_{K_W t_W K_W}$ in terms of the canonical retraction maps  
$\rho_{\cA_V, \cC_V}$ and $\rho_{\cA_W, \cC_W}$, respectively. The relation is obtained by considering \eqref{eq:hecke-adjV} (resp., \eqref{eq:hecke-adjW}) modulo the action of the unipotent radical of the Borel subgroup $B_V^{\ad} \subset G_V^{\ad}$ (resp., $B_W \subset G_W$) of upper-triangular matrices. One obtains

\begin{equation}\label{eq:satake-building}
\cS \circ |_{B_V^{\ad}} \mathbf{1}_{K_V^{\ad} t_V K_V^{\ad}} (x_0) = \sum_{\substack{x \in \cB(G_V^{\ad})\\  \{x : x_{0}\}= \{q^{a_1}, q^{a_{2}}, 1\}}} \rho_{\cA_V, \cC_V}(x), 
\end{equation}
and similarly, 
\begin{equation}\label{eq:satake-buildingW}
\cS \circ |_{B_W} \mathbf{1}_{K_Wt_W K_W} (y_0) = \sum_{\substack{y \in \cB(G_W)\\  \{y : y_{0}\}= \{q^{b_1}, q^{b_{2}}\}}} \rho_{\cA_W, \cC_W}(y), 
\end{equation}

Here, the canonical retractions are defined with respect to the chambers $\cC_V$ and $\cC_W$, i.e., $\rho_{\cA_V, \cC_V}(x)$ is the unique point $x'_V$ in $\cA_V$ so that $\dist (x_V', \cC_{V}) = \dist(x_V, \cC_{V})$.

\subsection{Counting images under the canonical retraction map}
Let $\cA_V$ be as in the previous section and let $x \in \cA_V$ be a ``neighbor" of the vertex $x_0$. The latter means that there exists a chamber containing both $x$ and $x_0$. To finish inverting the Satake transform, it remains to count the number of hyperspecial vertices in $\cB(G_V^{\ad})$ adjacent to $x_0$ that retract to $x$. To do this, we will restrict our attention to small neighborhoods of $x_0$ in $\cB(G_V^{\ad})$ called ``hexagons,'' showing that every point retracting to $x$ is contained in at least one hexagon and then counting the number of hexagons containing such a point.

\subsubsection{Hexagons and the main result}
For any apartment $\cA_V'$ containing $x_0$, the six chambers of $\cA_V'$ that have $x_0$ as a vertex (together with all their faces) form a \emph{hexagon}. We will call two apartments $\cA_V'$ and $\cA_V''$ containing $x_0$ \emph{equivalent} if their hexagons around $x_0$ coincide. We wish to count the number of hexagons containing the chamber $\cC_V$; equivalently, the number of equivalence classes of apartments.

Given the fixed vertex $x_0$ corresponding to the homothety class $[L_{x_0}]$ of the lattice $L_{x_0}$, one gets a coloring on the set of hyperspecial vertices of any apartment $\cA_V'$ containing $x_0$ that are neighbors of $x_0$ by the following rule: if $x$ is a neighbor of $x_0$ corresponding to a homothety class $[L_x]$ of lattices, consider the unique representative $L_{x}$, such that $qL_{x_0} \subsetneq L_x \subsetneq L_{x_0}$. 
We say that $x$ is \emph{even} if $L_x/qL_{x_0}$ is a plane in $L_{x_0}/qL_{x_0}$ and we say that $x$ is \emph{odd} if $L_{x}/qL_{x_0}$ is a line in $L_{x_0}/qL_{x_0}$. 
Note that if $x'$ and $x''$ are two vertices that share a common 2-simplex with $x_0$ in $\cB(G_V^{\ad})$ then $x'$ and $x''$ have opposite parity. We will unambiguously write $x_1$ for the odd vertex of $\cC_V$ and $x_6$ for the even vertex.

For the fixed apartment $\cA_V$, we label the six neighbors of $x_0$ by $x_1, \ldots, x_6$ counterclockwise (note that $x_1$, $x_6$ have already been labeled). We may now state the main result:
\begin{prop}\label{prop:count}
Under the canonical retraction map $\rho_{\cA_V, \cC_V}$,
\begin{itemize}
\item The vertices $x_1$ and $x_6$ each have a unique preimage (namely, $x_1$ and $x_6$, respectively).
\item The vertices $x_2$ and $x_5$ each have $q$ pre-images.
\item The vertices $x_3$ and $x_4$ each have $q^2$ pre-images.
\end{itemize}
\end{prop}

\begin{proof}
Let $H$ be an arbitrary hexagon containing $\cC_V$. To compute $\rho_{\cA_V, \cC_V}$ on the vertices of $H$, label the vertices of $H$ counter-clockwise $x_1, y_2, y_3, y_4, y_5, x_6$ ($x_1$ and $x_6$ are already labeled). Then $x_1$ and $x_6$ are fixed and $y_i \mapsto x_i$ for $2 \leq i \leq 5$ (since $\dist(y_i, \cC_V) = \dist(x_i, \cC_V)$ where $\dist$ denotes the gallery distance). Now, if $x$ is an arbitrary neighbor of $x_0$ in $\cB(G_V^{\ad})$ then (by the building axioms) there is an apartment $\cA_V'$ containing $\cC_V$ and $x$ and hence a hexagon $A = A(\cA_V')$ containing $\cC_V$ and $x$. Labeling $H$ with the notation schema above, $x$ is of one of the types $1, 2, \ldots, 6$, and it is clear that this type is independent of the choice of $\cA_V$ and $H$.

To prove the proposition, we count the total number of hexagons containing $\cC_V$ and then, for $i = 1, 2, \ldots, 6$, count the number of hexagons containing a fixed point of type $i$. The quotient then gives the number of distinct points of type $i$, which is the number of pre-images of $x_i$. The proposition is thus a direct consequence of Lemma~\ref{lem:count} below. 
\end{proof}

\begin{lem}\label{lem:count}
There are $q^3$ hexagons containing $\cC_V$. Of these hexagons: 
\begin{itemize}
\item all $q^3$ contain $x_1$ and $x_6$. 
\item for a neighbor $y_2 \in \cB(G_V^{\ad})$ of $x_0$ of type 2, there are $q^2$ hexagons containing $y_2$ and $\cC_V$.
\item for a neighbor $y_3 \in \cB(G_V^{\ad})$ of $x_0$ of type 3, there are $q$ hexagons containing $y_3$ and $\cC_V$.
\item for a neighbor $y_4 \in \cB(G_V^{\ad})$ of $x_0$ of type 4, there are $q$ hexagons containing $y_4$ and $\cC_V$.
\item for a neighbor $y_5 \in \cB(G_V^{\ad})$ of $x_0$ of type 5, there are $q^2$ hexagons containing $y_5$ and $\cC_V$.
\end{itemize}
\end{lem}

\begin{proof}
Note that a hexagon corresponds to a choice of three lines $\{ L_1, L_3, L_5 \}$ in the $\F$-vector space $L_{x_0} / qL_{x_0}$ that span that vector space (we call such a choice of lines a \emph{framing} of $L_{x_0} / qL_{x_0}$). The ``odd'' vertices correspond to the lines with the same index, and the ``even'' vertices correspond to the planes spanned by their neighbors. Write $\ell_1$ and $P_6 \supset \ell_1$ for the line (resp., the plane) of $L_{x_0} / qL_{x_0}$ determined by the vertix $x_1$ (resp., $x_6$). The total number of hexagons containing $\cC_V$ is equal to the number of distinct framings $\{ L_1, L_3, L_5 \}$ of $L_{x_0} / pL_{x_0}$ such that $L_1 = \ell_1$, $L_1 \oplus L_5 = P_6$, and $L_{x_0} / qL_{x_0} = L_1 \oplus L_3 
\oplus L_5$. There are $(q+1) - 1 = q$ choices for $L_5$ since $L_5$ is a line in the fixed plane $P_6$ different from $\ell_1$. Once $L_1$ and $L_5$ are fixed, we have $\displaystyle \frac{q^3 - 1}{q-1} - (q+1) = q^2$ choices for $L_3$, since $L_3$ can be any line not contained in $P_6$. Thus, there are $q^3$ hexagons containing $\cC_V$. 

Let $y_2$ be a fixed neighbor of type $2$ corresponding to a plane $P_2 \supset \ell_1$ in $L_{x_0} / qL_{x_0}$ that is different from $P_6$. A hexagon $\{ L_1, L_3, L_5\}$ contains $\cC_V$ and $y_2$ if and only if $L_1 = \ell_1$, $\langle L_1, L_5 \rangle = P_6$, and $\langle L_1, L_3 \rangle = P_2$. As before, there are $q$ possible choices for $L_5$ and once $L_1$ and $L_5$ are fixed, there are $q$ possible choices for $L_3$ in $P_2$ (since $P_2 \cap P_6 = \ell_1$, any line other than $\ell_1$ in $P_6$ will be independent from $\ell_1, L_5$). Thus, there are $q^2$ distinct hexagons containing $y_2$.

Let $y_3$ be a fixed type $3$ point, corresponding to a line $\ell_3$ in $L_{x_0} / qL_{x_0}$ not contained in $P_6$. Then a hexagon $\{L_1, L_3, L_5\}$ contains $\cC_V$ and $y_3$ if and only if $\ell_1 = L_1$, $\langle L_1, L_5 \rangle = P_6$, and $L_3 = \ell_3$. Since $L_1$ and $L_3$ are fixed by these conditions, the number of hexagons containing $\cC_V$ is the number of distinct choice for $L_5$ which is $q$.

The arguments for neighbors of type $4$ and $5$ follow from the arguments for type $2$ and $3$ points by the incidence-preserving duality between lines and planes in $L_{x_0} / qL_{x_0}$.
\end{proof}

\subsubsection{The inverse Satake isomorphism}\label{subsubsec:invertsatake}
To invert the Satake isomorphism on $G_V$ we first invert it on the level of $G_V^{\ad}$ and then use the isomorphism \eqref{eq:glpgl} and the 
fact that $\ds \widehat{\mathbf{1}_{\varpi K_V}} = u_V$. To invert the isomorphism on $G_V^{\ad}$ (isomorphic to $\PGL_3(k_0)$), we use  \eqref{eq:satake-building} and Proposition~\ref{prop:count}.   
Let $g', g'', g''' \in G_V$ be the elements corresponding to the diagonal matrices $\diag(\varpi, 1, 1)$, $\diag(\varpi,\varpi, 1)$ and $\diag(\varpi, \varpi, \varpi)$ under the isomorphism $G_V \isom \GL_3(k_0)$ (determined by the basis $\{e_1, e_2, e_3\}$). For $\bullet \in \{', '', '''\}$ let 
$t_{g^{\bullet}} \in \calH(G, K)$ denote the element $\mathbf{1}_{K (g^{\bullet}, 1) K}$. Similarly, define $h', h'' \in G_W$ to be the elements corresponding to 
$\diag(\varpi, 1)$ and $\diag(\varpi, \varpi)$ under $G_W \isom \GL_2(k_0)$ and $t_{h'}$ and $t_{h''}$ the corresponding elements in $\calH(G, K)$. Using \eqref{eq:satake-building} and Proposition~\ref{prop:count}, we obtain 
\begin{equation}
\cS \circ |_{B_V^{\ad}} \mathbf{1}_{K_V^{\ad} g' K_V^{\ad}}  =  \mathbf{1}_{g_{1} T_{V, c}^{\ad}}+ q\mathbf{1}_{g_{2} T_{V, c}^{\ad}}+ q^{2}\mathbf{1}_{g_{3} T_{V, c}^{\ad}}
\end{equation}
 and 
\begin{equation}
\cS \circ |_{B_V^{\ad}} \mathbf{1}_{K_V^{\ad} g'' K_V^{\ad}}  =  \mathbf{1}_{g_{1}g_{2}T_{V, c}^{\ad}}+ q\mathbf{1}_{g_{2}g_{3}T_{V, c}^{\ad}}+ q^{2}\mathbf{1}_{g_{3}g_{1} T_{V, c}^{\ad}},  
\end{equation}
where $g_1, g_2, g_3 \in T_V$ are the elements defined in Section~\ref{subsec:heckepoltorus} (and by abuse of notation, we also use those to denote the corresponding images in $T_V^{\ad}$). 
We now use these to compute the Satake tranform $\widehat{\cdot} \colon \mathcal H(G_V, K_V) \ra \mathcal H(T_V, T_{V, c})^{\Omega(T_V)}$ on the elements $t_{g'}, t_{g''}, t_{g'''}$: 
\[
\left | 
\begin{array}{l}
\widehat{t_{g'}}  = q( \mathbf{1}_{(g_{1}, 1)T_c}+ \mathbf{1}_{(g_{2}, 1) T_c}+ \mathbf{1}_{(g_{3}, 1) T_c}) = q s_{1, 0}\\ 
\widehat{t_{g''}}  = q( \mathbf{1}_{(g_{1}g_{2}, 1) T_c}+ \mathbf{1}_{(g_{2}g_{3}, 1) T_c}+ \mathbf{1}_{(g_{3}g_{1}, 1)T_c}) = q u_V s_{-1, 0}\\
\widehat{t_{g'''}} = \mathbf{1}_{(g_1g_2g_3, 1)T_c} = u_V.
\end{array}
\right .
\]
Note that the right-hand sides of the above three formulas are all invariant under the Weyl 
group, i.e., they lie in $\mathcal H(T, T_c)^{\Omega(T)}$. Hence 
\begin{equation}\label{eq:satakeV}
s_{1, 0} = q^{-1} \widehat{t_{g'}}, \qquad s_{-1, 0} = q^{-1} u_V^{-1} \widehat{t_{g''}}, \qquad \text{and} \qquad u_V = \widehat{t_{g'''}}.
\end{equation}

Similarly, we write the Satake transform for the small group $G_W$: setting $h' = \diag(\varpi, 1)$ and $h'' = \diag(\varpi, \varpi)$, we have 
\[
\left | 
\begin{array}{l}
\widehat{t_{h'}}  = q^{1/2} \left ( \mathbf{1}_{(1, h_1)T_c}+ \mathbf{1}_{(1, h_2) T_c} \right ) \\ 
\widehat{t_{h''}}  =  \mathbf{1}_{(1, h_1 h_2)T_c}  = u_W.   
\end{array}
\right .
\]
Using the above transformations as well as $u_W s_{0, -1} = s_{0, 1}$, we get 
\begin{equation}\label{eq:satakeW}
s_{0, 1} = q^{-1/2} \widehat{t_{h'}}, \qquad \text{and} \qquad 
s_{0, -1} = q^{-1/2}u_W^{-1} \widehat{t_{h'}}. 
\end{equation}
Finally, substituting \eqref{eq:satakeV} and \eqref{eq:satakeW} into the expression for $H_w(z)$ from Section~\ref{subsec:heckepoltorus}, we obtain the Hecke polynomial $H_w(z)$ with coefficients viewed as elements of the Hecke ring $\mathcal H_{\Z[q^{\pm 1/2}]}$ (in fact, the 
coefficients belong to the subring $\calH$):  

\begin{prop}\label{thm:heckeformula}
The Hecke polynomial $H_w(z)$ is the polynomial: 
\begin{eqnarray*}
H_w(z) &=& z^6 - t_{g''} t_{g'''}^{-1} t_{h'} t_{h''}^{-1}   z^5 + \\
&+& q ( t_{g'}t_{g'''}^{-1}t_{h'}^2t_{h''}^{-2} + t_{g''}^2t_{g'''}^{-2}t_{h''}^{-1} -2q t_{g'}t_{g'''}^{-1}t_{h''}^{-1}) z^4 - \\
&-& q^2 ( t_{g'}t_{g''} t_{g'''}^{-2}t_{h'} t_{h''}^{-2} - q t_{g'''}^{-1} t_{h'}^3 t_{h''}^{-3} + 3q^2 t_{g'''}^{-1}t_{h'}t_{h''}^{-2})z^3 + \\
&+& q^4(t_{g''} t_{g'''}^{-2} t_{h'}^2 t_{h''}^{-3} + t_{g'}^2 t_{g'''}^{-2} t_{h''}^{-2} -2q t_{g''}t_{g'''}^{-2} t_{h''}^{-2})z^2 - \\
&-& q^6 t_{g'} t_{g'''}^{-2} t_{h'} t_{h''}^{-3} z + \\ 
&+& q^9 t_{g'''}^{-2} t_{h''}^{-3}. 
\end{eqnarray*}
In particular, it is a polynomial in $\calH [z]$. 

\end{prop}

\section{Distribution Relations on Invariants}\label{sec:split-distrel-inv}
We keep the notation from Section~\ref{sec:split-orbits}. 
Let $\Inv$ be the set of $H_0$-invariants on $\Hyp_V \times \Hyp_W$ from Theorem~\ref{thm:orbits}B. Let $\Hyp = \Hyp_V \times \Hyp_W$ be the set of hyperspecial points on $G = G_V \times G_W$. 
Let 
\begin{equation}\label{eq:frobh}
h_{\Fr} = (\diag(\varpi^{-1}, 1, 1), \diag(\varpi^{-1}, 1)) \in H \subset G, 
\end{equation}
which is an element of $H$ that does not belong to $H_0$. The actions of both $h_{\Fr}$ and the local Hecke algebra $\calH = \calH(G, K)$ on $\Z[\Hyp]$ descend to actions on $\Z[\Inv]$ as explained in 
Section~\ref{subsec:action-inv}.  

\subsection{Distribution relations on invariants}

By Theorem~\ref{thm:orbits}B, the invariants $\Inv$ are identified with 6-tuples  
$(k, s, r, d, m, n)$ modulo the equivalence relation described in that theorem. Define the local conductor 
$$
\bc(k, s, r, d, m, n) = \max\{0, \min\{m - n, n+d-m\}\}. 
$$
Let 
$
H_w(z) = C_0 z^6 + \dots + C_6 \in \calH[z]
$ 
be the Hecke polynomial computed in Section~\ref{sec:split-hecke}. In this section, we prove the following distribution relations on $\Z[\Inv]$ which we then use in Section~\ref{sec:split-distrel-cycles} to deduce our main result: 

\begin{thm}\label{thm:distrel-inv}
The element $H_w(h_{\Fr}) \cdot (0, 0, 0, 0, 0, 0) \in \Z[\Inv]$ is supported on 6-tuples $(k, s, r, d, m, n)$ for which $\bc(k, s, r, d, m, n)$ is either $1$ or $\varpi$. Moreover, 
$$
H_w(h_{\Fr}) \cdot (0, 0, 0, 0, 0, 0) \in (q-1) \Z[\Inv].   
$$
\end{thm}   

\noindent We will prove the theorem by explicitly computing the action of $\mathcal H$ on $\Z[\Inv]$. 

\subsection{The action of $\calH$ and $h_{\Fr}$ on $\Z[\Inv]$}\label{subsec:action-inv}

Consider the action of the local Hecke algebra $\calH$ on $\Z[G / K]$. The $G$-action on $\cB(G)$ identifies $\Hyp$ with $G / K$ since $K$ is the stabilizer of the pair $(x_0, y_0)$ where 
$x_0$ corresponds to the hyperspecial maximal compact $K_{V} \subset G_V$ and $y_0$ to the hyperspecial maximal compact $K_{W} \subset G_W$.  
The decomposition $G =  B K$ shows that any $x \in \Hyp \isom G / K$ can be represented by an element 
of the Borel subgroup $B \subset G$. Let $x$ corresponds to a coset $b K$ for $b \in B$.  The Hecke operator 
$\mathbf{1}_{K t K}$ then acts on a hyperspecial vertex corresponding to $b K$ via 
$[bK] \mapsto \sum_{\alpha} [ bg_\alpha K ]$, 
where $\ds K t K = \bigsqcup_{\alpha} g_\alpha K$. It is easy to check that this action is well-defined. 

\begin{lem}
The action of the local Hecke algebra $\mathcal H$ on the set $\Z[\Hyp]$ descends to an action of $\mathcal H$ on $\Z[\Inv]$. 
\end{lem}

\begin{proof}
That the action of $\mathbf{1}_{K g K}$ on $\Z[\Hyp]$ descends to $\Z[\Inv]$ follows simply from 
the above description. For $h_{\Fr}$, we use that for any $h_0 \in H_0$ and any $(x, y) \in \Hyp$, 
$$
\inv(h_{\Fr} (h_0 x, h_0 y)) = \inv (h_0' h_{\Fr} x, h_0' h_{\Fr} y) = \inv (h_{\Fr} x, h_{\Fr} y), 
$$
where $h_0' \in H_0$ satisfies $h_0' = h_{\Fr^{-1}} h_0 h_{\Fr}$. 
\end{proof}

\subsection{Computing local invariants}\label{algorithm}
We summarize the construction of Section~\ref{subsec:Horbits} in an algorithm that gives the invariants $(k, s, r, d, m, n)$ for a pair $(\Lambda_V, \Lambda_W)$ of lattices in the $k_0$-vector spaces $V$ and $W$, respectively.

\begin{algorithm}\label{alg:invs}
\caption{Computing invariants of a pair of lattices $\Lambda_V$, $\Lambda_W$}
\begin{algorithmic}[1]
\REQUIRE $M \in \GL_3(k_0)$ and $N \in \GL_2(k_0)$ whose columnns are $\cO$-bases for $\Lambda_V$ and 
$\Lambda_W$, respectively.
\ENSURE A $6$-tuple $(k, s, r, d, m, n)$ for $\inv(\Lambda_V, \Lambda_W)$ as defined in Theorem \ref{thm:orbits}B.
\STATE Transform $M$ into an upper-triangular matrix by $\cO$-linear column operations.
\STATE Set $s := v_\varpi(M_{3, 3})$ and $M_0 := (M_{3, 3})^{-1} M$.
\STATE Write $\widetilde{M}$ for the upper left $2$-by-$2$-submatrix of $M_0$. 
\STATE Set $\widetilde{N} := \widetilde{M}^{-1} N$.
\STATE Perform a Cartan decomposition $\widetilde{N} = k_1 t k_2$, where $k_1, k_2 \in \GL_2(\cO)$ and $t$ is diagonal, with $v(t_{1, 1}) \geq v(t_{2, 2})$. 
\STATE Set $h := k_1 \widetilde{M}^{-1}.$
\STATE Set $U := h M_0$. As a consistency check, $U$ should be equivalent via $\cO$-linear column operations to a unitary matrix whose $(1, 2)$ entry is 0.
\RETURN $k = -v(\det h)$, $s$, $r = v(t_{2, 2})$, $d = v(t_{1, 1}) - r$, $m = -v(U_{1, 3})$, and $n = -v(U_{2, 3})$. (Recall that $m$ and $n$ are well-defined up to the equivalence of Proposition \ref{prop:RdOrbits}; set $m$, resp. $n$, to $0$, if $U_{1, 3} = 0$, resp. $U_{2, 3}= 0.$)
\end{algorithmic}
\end{algorithm}

\subsection{Computing the action of $\calH$ on $\Z[\Inv]$}
To compute the action of each $C_i$ on $\Z[\Inv]$, we use the canonical representative $b_\nu \in B$ for a given 6-tuple  
$\nu = (k, s, r, d, m, n)$ from Theorem~\ref{thm:orbits}. Recall that 
$$
b_\nu = \left ( \begin{pmatrix}\varpi^{s+k} & 0 & \varpi^{s+k-m} \\ 0 & {\varpi^s} & \varpi^{s-n} \\ 0 & 0 & {\varpi^s} \end{pmatrix}, \begin{pmatrix} \varpi^{k+d+r} & 0 \\ 0 & \varpi^r  \end{pmatrix} \right ) \in B_V \times B_W.
$$
We the consider the right coset decomposition of the double cosets 
$K_V g' K_V$ and $K_V g'' K_V$ as well as of $K_W h' K_W$ where $g', g''$ and $h'$ are the elements used in Section~\ref{subsubsec:invertsatake}. More precisely, 
\begin{equation}\label{eq:gprime}
K g' K = \bigsqcup_{a, b \in \F_q} b'_{a, b} K
\sqcup \bigsqcup_{c \in \F_q} b'_c K \sqcup b' K.
\end{equation}
Here, $\ds b'_{a, b} = \left ({\mthree \varpi {\widetilde{a}} {\widetilde{b}} {} 1 {} {} {} 1}, 1 \right)$ for any lifts  
$\widetilde{a}, \widetilde{b} \in \cO$ of $a$ and $b$, respectively. Moreover $\ds b'_c = \left ({\mthree 1 {} {} {} \varpi {\widetilde{c}} {} {} 1}, 1\right )$ for any lift $\widetilde{c} \in \cO$ of $c$ and $b' = (\diag(1, 1, \varpi), 1)$. 

Next, 
\begin{equation}\label{eq:gdoubleprime}
K g'' K = \bigsqcup_{a, b \in \F_q} b''_{a, b} K
\sqcup \bigsqcup_{c \in \F_q} b''_c K \sqcup b'' K, 
\end{equation}

where $\ds b''_{a, b} = \left ({\mthree \varpi {} {\widetilde{a}} {} \varpi {\widetilde{b}} {} {} 1}, 1 \right )$ for any lifts  
$\widetilde{a}, \widetilde{b} \in \cO$ of $a$ and $b$, respectively, $\ds b''_c = \left ( {\mthree \varpi {\widetilde{c}} {} {} 1 {} {} {} \varpi}, 1 \right )$ for any lift $\widetilde{c} \in \cO$ of $c \in \F_q$ and $b'' = (\diag(1, \varpi, \varpi), 1)$.
Finally, 
\begin{equation}\label{eq:doublecosethprime}
K_W h' K_W = \bigsqcup_{\widetilde{a} \in \F_q} {\mtwo {\varpi} {a} {} {1}} K_W \sqcup {\mtwo {1} {} {} {\varpi}} K_W, 
\end{equation}
where $\widetilde{a} \in \cO$ is any lift of $a \in \F_q$. 

Having these decompositions, we compute the action of the generating Hecke operators on invariants. The sequence of lemmas below recovers explicitly the action of $\mathcal H$ on $\Z[\Inv]$. The action of $h_{\Fr}$ is given by 
\begin{equation}\label{eq:frob}
h_{\Fr} \cdot (k, s, r, d, m, n) = (k-1, s, r, d, m, n). 
\end{equation}

\noindent We begin with describing the action of $t_{g'} = \mathbf{1}_{K g' K}$ via the following lemma together with the decomposition \eqref{eq:gprime}:  
\begin{lem}[action of $\mathbf{1}_{K g' K}$]\label{lem:gprime}
Let $\nu = (k, s, r, d, m, n)$. 

\noindent (i) If $a = 0$ then 
$$
\inv([b_\nu b'_{0, b}]) = (k+1, s, r', |d - 1|, m', n'), 
$$
where 
$$
r' = 
\begin{cases}
r & \text{if } d > 0,\\
r-1 & \text{if } d = 0, 
\end{cases}
$$
and 
$$
(m', n') = 
\begin{cases}
(0, n) & \text{if } d > 0, m = 0, b = -1,\\
(m+1, n) & \text{if } d > 0 \text{ and } (m = 0, b \ne -1 \text{ or } m\ne 0), \\
(n, 0) & \text{if } d = 0, m = 0, b = -1, \\ 
(n, m+1) & \text{otherwise.}
\end{cases}
$$
\noindent (ii) If $a \ne 0$ then 
$$
\inv([b_\nu b'_{a, b}]) = (k+1, s, r-1, d+1, m', n'), 
$$
where 
$$
(m', n') = 
\begin{cases}
(m, 0) & \text{if } m = n = 0 \text{ and } 1+ b - a = 0, \\ 
(m, 1) & \text{if } m = n > 0 \text{ and } a = 1, \\ 
(m, \max(m, n)) & \text{otherwise.}
\end{cases}
$$

\noindent (iii) One has  
$$
\inv([b_\nu b'_{c}]) = 
\begin{cases}
(k+1, s, r-1, d + 1, m, 0) & \text{if } n = 0, c = -1, \\ 
(k+1, s, r-1, d + 1, m, n+1) & \text{otherwise.}
\end{cases}
$$

\noindent (iv) One has 
$$
\inv([b_\nu b']) = (k-2, s+1, r+1, d, \max(m-1, 0), \max(n-1, 0)). 
$$
\end{lem}

\begin{proof}
The proof will be quite straightforward using Algorithm~\ref{alg:invs}. As such, we will only prove (i) and note that (ii)--(iv) will follow analogously. 
First, 
$$
b_\nu b_{0, b}' = \left ( \begin{pmatrix}\varpi^{s+k+1} & 0 & \varpi^{s+k} \widetilde{b}+ \varpi^{s+k-m} \\ 0 & {\varpi^s} & \varpi^{s-n} \\ 0 & 0 & {\varpi^s} \end{pmatrix}, \begin{pmatrix} \varpi^{k+d+r} & 0 \\ 0 & \varpi^r  \end{pmatrix} \right ).
$$
Since the matrix $\ds M = \begin{pmatrix}\varpi^{s+k+1} & 0 & \varpi^{s+k} \widetilde{b}+ \varpi^{s+k-m} \\ 0 & {\varpi^s} & \varpi^{s-n} \\ 0 & 0 & {\varpi^s} \end{pmatrix}$ is already upper-triangular, Step 2 of Algorithm~\ref{alg:invs} shows that the $s$-invariant does not change. We then compute the matrix $M_0 = \begin{pmatrix}\varpi^{k+1} & 0 & \varpi^{k} \widetilde{b}+ \varpi^{k-m} \\ 0 & 1 & \varpi^{-n} \\ 0 & 0 & 1 \end{pmatrix}$,  
$\widetilde{M} = \begin{pmatrix} \varpi^{k+1} & 0 \\ 0 & 1  \end{pmatrix}$ (Step 3) and (Step 4)
$$
\widetilde{N} = \widetilde{M}^{-1} N = 
\begin{pmatrix} 
\varpi^{-k-1} & 0 \\ 0 & 1
\end{pmatrix}
\begin{pmatrix} 
\varpi^{k+d+r} & 0 \\ 0 & \varpi^r  
\end{pmatrix} = 
\begin{pmatrix} \varpi^{d+r-1} & 0 \\ 0 & \varpi^r  \end{pmatrix}. 
$$
At this point, we need to perform the Cartan decomposition in Step 5. There are two cases 
to consider that will result in different matrices $k_1$: 

\vspace{0.1in}

\noindent {\bf Case 1: $d > 0$.} In this case $d+r-1 \geq r$ and hence, $k_1$ is the identity matrix and $t = \diag(\varpi^{d+r-1}, \varpi^r)$. Step 6 then yields 
$\ds h = 
\begin{pmatrix} 
\varpi^{-k-1} & 0 \\ 0 & 1
\end{pmatrix}
$
and Step 7 gives 
$$
U = 
\begin{pmatrix}
\varpi^{-k-1} & 0 & 0 \\
0 & 1 & 0 \\
0 & 0 & 1
\end{pmatrix}
\begin{pmatrix}
\varpi^{k+1} & 0 & \varpi^{k} \widetilde{b}+ \varpi^{k-m} \\ 
0 & 1 & \varpi^{-n} \\ 
0 & 0 & 1 
\end{pmatrix} = 
\begin{pmatrix}
1 & 0 & \varpi^{-1} \widetilde{b}+ \varpi^{-m-1} \\ 
0 & 1 & \varpi^{-n} \\ 
0 & 0 & 1 
\end{pmatrix}. 
$$
As a consistency check, $U$ is unipotent. We now read the invariants following Step 8: the 
$r$-invariant remains unchanged, whereas the $d$-invariant decreases by 1. We thus get (again following Step 8)
$$
\inv([b_\nu b_{0, b}]) = (k+1, s, r, d-1, m', n') = (k+1, s, r, |d-1|, m', n), 
$$
where  
$$
m' =  
\begin{cases}
0 & \text{if }m = 0, \ b = -1, \\
m+1 & \text{otherwise}.  
\end{cases}
$$

\vspace{0.1in}

\noindent {\bf Case 2: $d = 0$.} In this case we have $d + r - 1 < r$ and hence, our Cartan decomposition from Step 5 yields that 
$k_1 = k_2 = 
\begin{pmatrix} 
0 & 1 \\ 
1 & 0  
\end{pmatrix}$ and $t = \diag(\varpi^r, \varpi^{r-1})$. Step 6 then yields the matrix 
$$
h = 
\begin{pmatrix} 
0 & 1 \\ 
1 & 0  
\end{pmatrix} 
\begin{pmatrix} 
\varpi^{-k-1} & 0 \\ 
0 & 1
\end{pmatrix} = 
\begin{pmatrix}
0 & 1 \\
\varpi^{-k-1} & 0
\end{pmatrix}. 
$$ 
Step 7 then yields the matrix 
$$
U = 
\begin{pmatrix}
0 & 1 & 0 \\
\varpi^{-k-1} & 0 & 0 \\
0 & 0 & 1
\end{pmatrix}
\begin{pmatrix}
\varpi^{k+1} & 0 & \varpi^{k} \widetilde{b}+ \varpi^{k-m} \\ 
0 & 1 & \varpi^{-n} \\ 
0 & 0 & 1 
\end{pmatrix} = 
\begin{pmatrix}
0 & 1 & \varpi^{-n} \\ 
1 & 0 & \varpi^{-1} \widetilde{b}+ \varpi^{-m-1} \\ 
0 & 0 & 1 
\end{pmatrix}, 
$$
which is unipotent. Step 8 then yields all the invariants: 
$$
\inv([b_\nu b_{0, b}]) = (k+1, s, r-1, 1, m', n') = (k+1, s, r-1, |d-1|, n, n'), 
$$
where 
$$
n = 
\begin{cases}
0 & \text{if } m = 0, \ b = -1 \\ 
m+1 & \text{otherwise}. 
\end{cases} 
$$ 

\noindent Finally, (i) easily follows from putting together the two cases. 
\end{proof}

\noindent The following lemma whose proof is exactly the same as the proof as the proof of Lemma~\ref{lem:gprime}, together with equation \eqref{eq:gdoubleprime}, recovers the action of $t_{g''} = \mathbf{1}_{K g'' K}$ on $\Z[\Inv]$:  

\begin{lem}[action of $\mathbf{1}_{K g'' K}$]\label{lem:gdoubleprime}
Let $\nu = (k, s,r , d, m, n)$. 

\noindent (i) One has  
$$
\inv([b_\nu b''_{a, b}]) = 
(k+2, s, r-1, d, m', n'),
$$
where 
$$
m' = 
\begin{cases}
0 & \text{if } m = 0, a = -1, \\
m+1 & \text{otherwise.}
\end{cases}
$$
and 
$$
n' = 
\begin{cases}
0 & \text{if } n = 0, b = -1, \\
n+1 & \text{otherwise.}
\end{cases}
$$ 

\noindent (ii) If $c = 0$, one has  
$$
\inv([b_\nu b''_{0}]) = 
\begin{cases}
(k-1, s+1, r, d+1, \max(n-1, 0), m) & \text{if } d = 0, \\ 
(k-1, s+1, r+1, d-1, m, \max(n-1, 0)) & \text{if } d \ne 0. 
\end{cases}
$$

\noindent (iii) If $c \ne 0$, one has 
$$
\inv([b_\nu b''_{c}]) = 
\begin{cases}
(k-1, s+1, r, d+1, \max(n-1, 0), 0) & \text{if } d = 0, m = n, c=1, \\ 
(k-1, s+1, r, d+1, \max(n-1, 0), \max(m, n)) & \text{if } d = 0 \text{ and } (c \ne 1 \text{ or } m\ne n), \\  
(k-1, s+1, r-1, d+1, \min(m-1, 0), 0) & \text{if } d \ne 0, m=n, a=1, \\ 
(k-1, s+1, r-1, d+1, \min(m-1, 0), \max(m, n)) & \text{otherwise.}
\end{cases}
$$

\noindent (iv) One has 
$$
\inv([b_\nu b'']) = 
(k-1, s+1, r, d+1, \min(m-1, 0), n). 
$$
\end{lem}

\noindent Next, the action of the Hecke operator $t_{h'} = \mathbf{1}_{K h' K}$ is determined by the following lemma together with the double coset decomposition \eqref{eq:doublecosethprime}: 

\begin{lem}[action of $\mathbf{1}_{K h' K}$]
Let $\nu = (k, s, r, d, m, n)$. 

\noindent (i) If $b_W = \left ( 1, {\mtwo {\pi} {\widetilde{a}} {0} {1}}\right )$ then 
$$
\inv([b_\nu b_W]) = 
\begin{cases}
(k, s, d+1, r, m, n) & \text{if } a = 0 \\ 
(k, s, d+1, r, 0, n) & \text{if } m = n - d,\ a = 1, \\ 
(k, s, d+1, r, max(m, n-d), n) & \text{otherwise}.  
\end{cases}
$$

\noindent (ii) If $b_W = \left ( 1, {\mtwo {1} {0} {0} {\pi}}\right )$ then 
$$
\inv([b_\nu b_W]) = 
\begin{cases}
(k, s, r, d+1, n, m) & \text{if } d = 0, \\ 
(k, s, r+1, d-1, m, n) & \text{otherwise.}
\end{cases}
$$
\end{lem}

\noindent Finally, we record the action of the operators $t_{g'''} = \mathbf{1}_{K g''' K}$ and $t_{h''} = \mathbf{1}_{K h'' K}$: 
\begin{equation}\label{eq:gtprime}
\mathbf{1}_{Kg''' K} (k, s, r, d, m, n) = (k, s+1, r, d, m, n). 
\end{equation}
and 
\begin{equation}\label{eq:hdprime}
\mathbf{1}_{K h'' K} (k, s, r, d, m, n) = (k, s, r+1, d, m, n). 
\end{equation}

\subsection{Proof of Theorem~\ref{thm:distrel-inv}}\label{subsec:explicit}
Using the explicit computation of the Hecke action as well as the explicit formula for the Hecke polynomial, we compute  
\begin{eqnarray*}
H(h_{\Fr})\cdot (0, 0, 0, 0, 0, 0) &=& (q - 1)^2 q^2 (q^2 + q - 1) (-3, 0, 0,1,0,1) + \\ 
&+& (q - 1)^2 q (q^6 + 2q^5 + 2q^4 + 2q^3 - q^2 + 4) (0, 0, 0, 0, 0, 0) + \\ 
&+& (q-1)^2 q^3 (q+1)^2 (-3, 0, 0, 3, 0, 0) - \\ 
&-&  (q - 1)^3 q^2 (q + 1) (q^3 + q^2 - 1) (0, 0, 0, 2, 0, 1) - \\ 
&-& (q - 1)^2 (q^5 + 2q^4 - 2q^2 + q + 1) (-3, 0, 0, 1, 0, 0) - \\
&-&  (q - 1) q (q^3 - q + 1) (q^3 + q^2 + q - 2)(0, 0, 0, 0, 1, 0) - \\
&-& (q - 1)^2 q^3 (q + 1)^3(0, 0, 0, 2, 0, 0) - \\ 
&-& (q - 1)^3 (q + 1)^2  (-3, 0, 0, 1, 1, 0) + \\ 
&-& (q-1)q^4 (0, 0, 0, 0, 2, 0) - \\ 
&+& (q-1)q^3 (0, 0, 0, 0, 0, 2)  - \\
&-&  (q-1)^2 q(q+1)(-6, 0, 0, 0, 0, 0) - \\
&-&  (q - 1)^2 q^3 (q + 1) (-3, 0, 0, 3, 0, 1)  + \\
&+&   (q - 1)^3 q^3 (q + 1)^2 (0, 0, 0, 2, 0, 2) + \\ 
&+&   (q - 1) q (q^3 - 2q + 2) (0, 0, 0, 0, 0, 1) + \\ 
&+&  (q - 1)^2 q^2 (q + 1) (q^2 + q + 1) (0, 0, -1, 2, 1, 0) - \\
&-&  (q-1)^2 q (q+1)(-6, 0, 0, 2, 0, 0). 
\end{eqnarray*}
It remains to check that for each $\nu \in \Inv$ that is in the support of the above cycle, $\bc(\nu)$ is either $0$ or $1$, which is immediate from the definition of $\bc(\nu)$. 

\section{Distribution Relations on Special Cycles}\label{sec:split-distrel-cycles}
We return to the global notation from the first two sections. Theorem~\ref{thm:distrel-inv} will now be used to deduce Theorem~\ref{theoremC} (the horizontal distribution relations for $\cZ_K(\G, \Hbf)$). Recall that $\tau$ is an allowable prime of $F$ that is split in $E$ and that $\tau = w \overline{w}$ for places $w$ and $\overline{w}$ of $E$, with $w$ the place of $E$ corresponding to the fixed embedding $\iota_\tau \colon \overline{E} \hra \overline{F}_\tau$ and let 
$$
\Art_{E_w} \colon E_w^\times \ra \Gal(E_w^{\ab} / E_w)
$$ 
be the corresponding local Artin map. The fixed embedding $\iota_\tau$ then identifies the local Galois group $\Gal(E_w^{\ab} / E_w)$ with the decomposition group  
$D_w \subset \Gal(E^{\ab} / E)$ at the unique place of $E^{\ab}$ determined by the fixed embedding $\iota_\tau$. 

\subsection{Action of the decomposition group at $w$ on $\cZ_K(\G, \Hbf)$}
Let $\xi \in \cZ_K(\G, \Hbf)$ be a cycle of local conductor $\bc_\tau(\xi) = n$. 
Since 
$$
\ds \calH(\G, K) \isom \calH(G_\tau, K_\tau) \otimes \calH(G^{(\tau)}, K^{(\tau)})
$$
and  
$$
\G(\Af)/K \isom G_\tau / K_\tau \times G^{(\tau)} / K^{(\tau)} \isom \Hyp_\tau \times G^{(\tau)} / K^{(\tau)},
$$ 
to prove Theorem~\ref{theoremC}, it suffices to compare the action of the local Hecke algebra $\calH(G_\tau, K_\tau)$ to the action of the local Galois group $\Gal(E^{\ab}_w / E_w)$. 

Given integers $0 \leq m < n$, define the local trace at $\xi$ as.
$$
\Tr_{n, m}(\xi) := \sum_{x \in \cO_{m}^\times / \cO_{n}^\times} \xi^{\Art_{E_w}(x)}.   
$$
By local class field theory, the above trace computes $\Tr_{E(\tau^n)_w / E(\tau^m)_w} \xi$.  

\subsection{The action of $\Fr_w$}
Let $\xi \in \Z[\cZ_K(\G, \H)]$ be a cycle whose local conductor at $\tau$ is $\bc_\tau(\xi) = 0$. Recall from \textsection~\ref{subsec:frob} that $\Fr_w$ acts on $\xi$ via the 
element 
\begin{equation}\label{eq:frobh}
h_{\Fr} = (\diag(\varpi^{-1}, 1, 1), \diag(\varpi^{-1}, 1)) \in H_\tau \subset G_\tau, 
\end{equation}

Applying $h_{\Fr}$ to a cycle $\xi \in \cZ_K(\G, \H)$ with $\inv(\xi) = (k, s, r, d, m, n)$ of local conductor $\bc_\tau(\xi) = 0$, one has
\begin{equation}
\inv(\Fr_w \xi) = (k-1, s, r, d, m, n).
\end{equation}
(note that this action would not be well-defined if one had $\bc_\tau(\xi) > 0$).
\subsection{Galois orbits and invariants}

Fix a cycle $\xi_0 = \cZ_K(g_0)$ for some fixed $g_0 \in \G(\Af)$ defined over $L = E(\xi_0)$; if necessary, enlarge $L$ to contain $E(1)$. Suppose that $(g_0)_\tau = 1$. Then there is a map
$$
\Phi_{g_0} \colon G_\tau / K_\tau \to \cZ_K(\G, \H), \qquad g_\tau K_\tau \mapsto \cZ_K(g_\tau, g_0^{(\tau)}),  
$$ 
where $(g_\tau, g_0^{(\tau)}) \in \G(\Af)$ is the adelic element that agrees with $g_0$ outside the finite place $\tau$ and that is equal to $g_\tau$ at $\tau$. Composing this map with the map $\cZ_K(\G, \H) \twoheadrightarrow \Gal(E^{\ab} / E) \backslash \cZ_K(\G, \H)$ that takes a cycle to its Galois orbit, we obtain an induced map:  

$$
\Phi_{g_0} \colon H_\tau \backslash G_\tau / K_\tau \to D_w \backslash \Phi_{g_0}(G_\tau / K_\tau),  
$$
where $D_w \subset \Gal(E^{\ab} / E)$ is the decomposition group.  

Similarly, there is a map 
$$
\Phi_{g_0} \colon H_0 \backslash G_\tau / K_\tau \ra D_w^0 \backslash \Phi_{g_0}(G_\tau / K_\tau), 
$$
where $D_w^0$ is the decomposition group at $\iota_\tau$ of $\Gal(E^{\ab} / E(1))$. 

\begin{lem}\label{inv:cycle}
Assume that $L_w = E(1)_w$ (i.e., the place of $E(1)$ determined by $\iota_\tau$ splits completely in $L$). Let $\xi = \Phi_{g_0}(x)$ and $\xi' = \Phi_{g_0}(x')$ for $x, x' \in G_\tau / K_\tau$. If $D(L)_w \subset \Gal(E^{\ab} / L)$ denotes the decomposition group at $\iota_\tau$ then 
$$
\xi' \in D(L)_w \cdot \xi \qquad \text{ if and only if } \qquad \inv_\tau (x) = \inv_\tau (x').
$$ 
\end{lem}

\begin{proof}
By Theorem~\ref{thm:orbits}.B the elements $x$ and $x'$ are in the same $H_\tau$-orbit if and only if 
$\inv_\tau(x) = \inv_\tau(x')$. The claim is then that $\Phi_{g_0}(x)$ and $\Phi_{g_0}(x')$ are conjugate under $\Gal(E^{\ab}_w / L_w)$ if and only if the elements $x$ and $x'$ are in the same $H_\tau$-orbit. This follows from the description of the Galois action on the cycles $\cZ_K(\G, \Hbf)$ in \cite[\textsection 2.3.16]{jetchev:unitary}: the Galois group $\Gal(E^{\ab} / E)$ acts via $\Hbf(\Af)$ and the decomposition group $D_w^0 \subset \Gal(E^{\ab} / E(1))$ acts via $H_0$. 
\end{proof}

\subsection{Horizontal distribution relations on $\Z[\cZ_K(\G, \Hbf)]$}
Let $\xi_0 \in \cZ_K(\G, \H)$ be the special cycle from the statement of the Theorem~\ref{theoremC} for which 
$$
\inv_\tau(\xi_0) = (0, 0, 0, 0, 0, 0).
$$  
Then

\begin{eqnarray*}\label{eq:expand}
H_{w}(\Fr_w) \cdot \xi_{0} &=& \sum_{\xi \in \cZ_K(\G, \H)} C(\xi) \xi = \sum_{\nu \in \Inv_\tau} 
\sum_{\substack{\xi \in \cZ_K(\G, \H), \\ \inv_\tau(\xi) = \nu}} C(\xi) \xi \in \Z[\cZ_K(\G, \H)]. \\
\end{eqnarray*}

For $\nu \in \Inv_\tau$, write $c(\nu)$ for the local conductor of any cycle with invariant $\nu$. 
Moreover, fix a cycle $\xi_\nu \in \cZ_K(\G, \H)$ with $\inv_\tau(\xi_\nu) = \nu$. 
For each cycle $\xi$ with $\inv_\tau(\xi) = \nu$ in the above sum, there exists $x \in  \cO_0^\times / \cO_{c(\nu)}^\times$ such that $\xi = \xi_\nu^{\Art_{E_w}(x)}$. In addition, 
the quotient group $\cO_0^\times / \cO_{c(\nu)}^\times$ acts simply transitively on the local orbit $D_w^0 \xi_\nu$ under the decomposition group $D_w^0 \subset \Gal(E^{\ab} / E(1))$.  
Regrouping the terms of $x$ then yields integers $m_\nu(x) \in \Z$ such that 

\begin{equation}\label{eq:distrel}
H_w(\Fr_w) \cdot \xi_0 = \sum_{\nu \in \Inv_\tau} \sum_{x \in \cO_0^\times / \cO_{c(\nu)}^\times}  m_{\nu}(x)\xi_{\nu}^{\Art_{E_w}(x)}.
\end{equation}
By Theorem~\ref{thm:distrel-inv}, for all cycles $\xi \in \Supp (H_{w}(\Fr_w) \cdot \xi_{0})$, $\bc_w(\xi) = 0$ or $1$. From the explicit formula of Section~\ref{subsec:explicit}, the only value of 
$\nu \in \Supp(H_w(h_{\Frob}) \cdot (0, 0, 0, 0, 0, 0))$ for which $c(\nu) = 1$ is $\nu = (0, 0, -1, 2, 1, 0)$. We would like to show that the right-hand side of \eqref{eq:distrel} is of the form $\Tr_{1, 0} (\xi_1)$ for some element $\xi_1 \in \Z[\cZ_K(\G, \H)]$ with $\bc_w(\xi_1) = 1$. For each $\nu \ne (0, 0, -1, 2, 1, 0)$, the fact that 
$$
H_w(h_{\Fr}) \cdot (0, 0, 0, 0, 0, 0) \in (q-1) \Z[\Inv_\tau]
$$ 
implies that the summand in the right-hand side corresponding to $\nu$ is in the image of $\Tr_{1, 0}$. It remains to show that 
\begin{equation}\label{eq:sumconj}
\sum_{x \in \cO_0^\times / \cO^\times_{c(0, 0, -1, 2, 1, 0)}} m_\nu(x) \xi_\nu^{\Art_{E_w}(x)} 
\end{equation}
is in the image of the trace map $\Tr_{1, 0}$. The left-hand side of \eqref{eq:distrel} is clearly 
invariant under $D_w^0$ since the coefficients of the Hecke polynomial (being local Hecke operators) commute with the action of $D_w^0$ and since $\xi_0$ is invariant under $D_w^0$.  
Hence, the right-hand side of the same equation is invariant under $D_w^0$ too, i.e., \eqref{eq:sumconj} must be invariant under $D_w^0$. As the $\xi_\nu^{\Art_{E_w}(x)}$'s are all the distinct $D_w^0$-conjugates, we obtain that $m_{(0, 0, -1, 2, 1, 0)}(x)$ are all equal for $x \in \cO_0^\times / \cO_1^\times$ and \eqref{eq:sumconj} is in the image of $\Tr_{1, 0}$. This proves Theorem~\ref{theoremC}.

\section*{Acknowledgements}
This research project was funded by the Swiss National Science Foundation grant PP00P2\_144658. We thank Christophe Cornut and Christopher Skinner for various helpful discussions. 

\bibliographystyle{amsalpha}
\bibliography{biblio-math}

\end{document}